\documentclass [a4paper,10pt]{amsart}
\usepackage{longtable}
\usepackage{hyperref}
\usepackage[T1]{fontenc}
\usepackage[utf8]{inputenc}
\usepackage[english]{babel}
\usepackage{textcomp}
\usepackage{dsfont}
\usepackage{latexsym}
\usepackage{amssymb}
\usepackage{amsthm}
\usepackage{amsmath}
\DeclareMathAlphabet{\mathpzc}{OT1}{pzc}{m}{en}
\usepackage{yfonts}
\usepackage{xfrac}
\usepackage{newlfont}
\usepackage{graphicx}
\usepackage{mathtools}
\usepackage{comment}
\usepackage{indentfirst}
\usepackage{braket}
\usepackage{mathrsfs}

\newcommand{\itb}{\item[{\tiny $\bullet$}]}
\newcommand{\vb}{\vect b}
\newcommand{\vd}{\vect d}
\newcommand{\vm}{\vect m}
\newcommand{\vs}{\vect s}

\newcommand{\atomic}{(L)}
\newcommand{\atomics}{(L')}

\usepackage{etoolbox}
\apptocmd{\lim}{\limits}{}{}
\apptocmd{\sup}{\limits}{}{}
\apptocmd{\inf}{\limits}{}{}
\apptocmd{\liminf}{\limits}{}{}
\apptocmd{\limsup}{\limits}{}{}
\pretocmd{\langle}{\left}{}{}
\pretocmd{\rangle}{\right}{}{}

\usepackage{scalerel}[2014/03/10]
\usepackage[usestackEOL]{stackengine}
\newcommand{\dashint}{\,\ThisStyle{\ensurestackMath{%
\stackinset{c}{.2\LMpt}{c}{.5\LMpt}{\SavedStyle-}{\SavedStyle\phantom{\int}}}%
\setbox0=\hbox{$\SavedStyle\int\,$}\kern-\wd0}\int}

\DeclareMathOperator{\card}{Card}

\DeclareMathOperator{\supp}{Supp}

\DeclareMathOperator{\tr}{Tr}

\DeclareMathOperator{\Hol}{Hol}

\DeclareMathOperator{\Ima}{Im}
\DeclareMathOperator{\Rea}{Re}

\newcommand{\Supp}[1]{\supp\left( #1\right) }

\newcommand{\ee}{\mathrm{e}}

\newcommand{\vect}[1]{\mathbf{{#1}}}
\newcommand{\dd}{\mathrm{d}}

\DeclarePairedDelimiter{\abs}{\lvert}{\rvert}

\DeclarePairedDelimiter{\norm}{\lVert}{\rVert}

\let\originalleft\left
\let\originalright\right
\renewcommand{\left}{\mathopen{}\mathclose\bgroup\originalleft}
\renewcommand{\right}{\aftergroup\egroup\originalright}

\newcommand{\N}{\mathds{N}}

\newcommand{\C}{\mathds{C}}

\newcommand{\R}{\mathds{R}}

\newcommand{\Ms}{\mathscr{M}}

\newcommand{\Ff}{\mathfrak{F}}

\newcommand{\Cc}{\mathcal{C}}

\newcommand{\Ec}{\mathcal{E}}
\newcommand{\Fc}{\mathcal{F}}

\newcommand{\Hc}{\mathcal{H}}
\newcommand{\Ic}{\mathcal{I}}

\newcommand{\cM}{\mathcal{M}}
\newcommand{\Nc}{\mathcal{N}}

\newcommand{\Sc}{\mathcal{S}}

\newcommand{\meg}{\leqslant}
\newcommand{\Meg}{\geqslant}
\newcommand{\eps}{\varepsilon}
\renewcommand{\phi}{\varphi}
\newcommand{\mi}{\mu}

\newcommand{\Lin}{\mathscr{L}}

\begin{document}

\numberwithin{equation}{section} 
\theoremstyle{definition}
\newtheorem{deff}{Definition}[section]

\newtheorem{oss}[deff]{Remark}

\newtheorem{ass}[deff]{Assumptions}

\newtheorem{nott}[deff]{Notation}

\theoremstyle{plain}
\newtheorem{teo}[deff]{Theorem}

\newtheorem{lem}[deff]{Lemma}

\newtheorem{prop}[deff]{Proposition}

\newtheorem{cor}[deff]{Corollary}

\title[Toeplitz and Cesàro Operators on Siegel Domains]{Toeplitz and Cesàro-Type Operators on Homogeneous Siegel Domains} 
\author[M. Calzi, M. M. Peloso]{Mattia Calzi, Marco M. Peloso}

\address{Dipartimento di Matematica, Universit\`a degli Studi di
Milano, Via C. Saldini 50, 20133 Milano, Italy}
\email{{\tt mattia.calzi@unimi.it}}
\email{{\tt marco.peloso@unimi.it}}

\thanks{{\em Math Subject Classification 2020:} Primary: 32A36; Secondary: 32A10, 32M10. }
\thanks{Both authors are members of the
Gruppo Nazionale per l'Analisi Matematica, la Probabilit\`a e le
loro 
Applicazioni (GNAMPA) of the Istituto Nazionale di Alta Matematica
(INdAM)}

\begin{abstract}
In this paper we study Toeplitz and Cesàro-type operators on 
holomorphic function spaces on a homogeneous Siegel domain of Type
II.  We prove several necessary conditions and sufficient conditions
for these operators to be continuous or compact, or to belong to
suitable Schatten classes. 
\end{abstract}
\maketitle

\section{Introduction}

In this paper we study various mapping properties of Toeplitz and
Cesàro-type operators between mixed-norm weighted Bergman spaces on
homogeneous Siegel domains. Let us first introduce Toeplitz and
Cesàro(-type) operators.

Let $U$ be the unit disc in $\C$, and denote by $H^p(U)$ the Hardy
space on $U$ of type $L^p$, $p\in ]0,\infty]$. Then, $H^2(U)$ is a
reproducing kernel hilbertian space, and its reproducing kernel is
given by 
\[
K(z,w)= c (1-z\overline w)^{-1}
\]
for $z,w\in U$, for a suitable constant $c>0$.  The corresponding projector
\[
S f(z)\coloneqq c \int_{T} f(w)(1-z\overline w)^{-1} \,\dd w
\]
then induces continuous linear mappings of $L^p(T)$ onto $H^p(U)$ for
every $p\in ]1,\infty[$, where $T$ denotes the boundary of $U$.  
Given $g\in L^\infty(T)$,  one may then consider the Toeplitz operator
$f\mapsto S(f g)$ of $L^p(T)$ into $H^p(U)$, $p\in ]1,\infty[$. It
turns out that the matrix of the restriction of such operator to
$H^p(U)$ with respect to the standard monomial basis
$(z^k)_{k\in \N}$ has constant diagonal coefficient, i.e., is a
Toeplitz matrix. Conversely, an endomorphism of $H^2(U)$ whose matrix
with respect to the basis $(z^k)_{k\in \N}$ is Toeplitz is necessarily
of the form $f\mapsto S(f^* g)$ for some $g\in L^\infty(T)$, where
$f^*$ is the boundary value function associated with $f$ (cf.,
e.g.,~\cite[Theorem
3.2.6]{MartinezRosenthal}). See~\cite{MartinezRosenthal} and the
references therein for more details about  Toeplitz operators on the
Hardy space. 

One may then extend the preceding family of operators to more general reproducing kernel hilbertian spaces. Let us briefly discuss the case of weighted Bergman spaces.
Given $s>-1$ and $p\in ]0,\infty]$, it is known that the weighted Bergman spaces
\[
A^p_s(U)\coloneqq \Set{f\in \Hol(U)\colon \int_U \abs{f(z)}^p (1-\abs{z}^2)^s \,\dd z  <\infty}
\]
(modification for $p=\infty$) are quasi-Banach spaces and embed continuously into $\Hol(U)$. In particular, $A^2_s(U)$ is a reproducing kernel hilbertian space, with reproducing kernel
\[
K_s(z,w)= c_s (1-z\overline w)^{-2-s}
\]
for a suitable constant $c_s>0$. The corresponding (weighted Bergman) projector
\[
P_s f(z)\coloneqq c_s \int_U f(w)(1-z\overline w)^{-2-s}(1-\abs{w}^2)^s \,\dd w
\]
then induces continuous endomorphisms of $L^p_s(U)$ for every $p\in ]1,\infty[$. 
Given $g\in \Hol(U)$ and $s'>0$, the the compression with $P_{s'}$ of the operator of multiplication by $g$, namely,
\[
f \mapsto P_{s'}(f g) := T_g f
\]
is called a Toeplitz operator with symbol $g$.
More generally, given a Radon measure $\mi$ on $U$, one may consider the operator
\[
f \mapsto  \int_U f(w)(1-\,\cdot\,\overline w)^{-2-s}\,\dd \mi(w) 
:= T_\mi f,
\]
and still call it a Toeplitz operator  with  symbol $\mi$. 
One may then consider Toeplitz operators between weighted Bergman
spaces on more general
domains. Cf.~\cite{Luecking7,Zhu3,LiLuecking,MiaoZheng,Zhu,Zhu2,ChoeKooLee,Luecking8,Constantin,NanaSehba,PelaezRattyaSierra,AbateMongodiRaissy}
and the references therein for more details on various aspects of the
theory of Toeplitz operators on Bergman spaces.

It is known that the monomials $z^k$, $k\in \N$, form an
orthonormal basis for the Hardy space $H^2(U)$. Therefore, the space
$H^2(U)$ may be identified with $\ell^2(\N)$. Therefore, the Cesàro
operator on $\ell^2(\N)$ (cf.~\cite[326]{HLP}) 
\[
\ell^2(\N)\ni \lambda \mapsto \left( \frac{1}{k+1}\sum_{j\meg k} \lambda_j  \right)_{k\in \N}\in \ell^2(\N),
\]
can be transferred to an endomorphism of $H^2$, 
\[
\Cc\colon H^2(U)  \ni f\mapsto \int_0^{\,\cdot\,} \frac{f(w)}{1-w}\,\dd w\in H^2(U).
\]
As observed in~\cite{AlemanSiskakis}, this operator can be considered as a particular case (corresponding to the choice $g\coloneqq- \log(1-\,\cdot\,)$) of the operator
\[
\Cc_g\colon  f \mapsto \int_0^{\,\cdot\,} f(w) g'(w)\,\dd w,
\]
for $g\in \Hol(U)$.
The operators $\Cc_g$ may then be investigated on more general spaces. In~\cite{AlemanSiskakis}, the mapping properties of the operators $\Cc_g$ between various weighted Bergman space were investigated. Cf.~\cite{Siskakis,Aleman,AlemanConstantin,Constantin} for more details on these operators.

As noted in~\cite{NanaSehba}, the operator $\Cc_g$ may be essentially characterized by the property
\[
(\Cc_g f)'= f g',
\] 
and then extended to general homogeneous Siegel domains replacing the
standard derivative with more general Riemann--Liouville operators. We
call the resulting operators `Cesàro-type operators'. Notice, though,
that this latter interpretation basically reduces the study of such
operators to the study of the corresponding multiplication operators,
since the precise definition of $\Cc_g f$ is no longer specified and
we content ourselves with defining $\Cc_g f$ modulo the kernel of the
chosen Riemann--Liouville operator.

We shall now briefly describe homogeneous Siegel domains and the weighted Bergman spaces thereon.
Fix a complex hilbertian space $E$  of finite dimension $n$, a real
hilbertian space $F$  of finite dimension $m>0$, and an open convex cone $\Omega$ 
in $F$ which does  not contain any affine lines. $\Omega$ is said to be homogeneous if the group $G(\Omega)$ of its linear automorphisms acts transitively on it.\footnote{We shall generally
denote by $\langle\,\cdot\,,\,\cdot\,\rangle$ bilinear pairings and
real scalar products, and by $\langle\,\cdot\,\vert\,\cdot\,\rangle$
sesquilinear pairings and complex scalar products, without
specifying the involved spaces.} 
Take a non-degenerate hermitian
mapping $\Phi\colon E\times E\to F_\C $  such that $\Phi(\zeta)\coloneqq \Phi(\zeta,\zeta)\in
\overline\Omega$ for all $\zeta\in E$.  Then, the  Siegel domain of
type II associated with the cone $\Omega$ and the mapping $\Phi$ is
\[
D \coloneqq \Set{ (\zeta, z) \in E \times F_\C \colon
\rho(\zeta,z)\coloneqq \Ima z -  \Phi(\zeta) \in \Omega } .
\]
When $n=0$, i.e., $E=\Set{0}$, $D$ is said to be a Siegel domain of Type I, or  a {\em  tubular domain} over the cone $\Omega$.  
The domain $D$ is   homogeneous if the group of its biholomorphisms acts transitively on $D$, in which case the group of its \emph{affine} automorphisms acts transitively (cf., e.g.,~\cite[Theorem 2.3]{Xu}). More
precisely, $D$ is homogeneous if and only if for every $h,h'\in
\Omega$ there are $t\in G(\Omega)$ and $g\in GL(E)$ such that $th=h'$ and
such that $t\Phi=\Phi(g\times g)$ (so that $g\times t$ preserves
$D$), cf., e.g.,~\cite[Propositions 2.1 and 2.2]{Murakami}. In
particular, if $D$ is homogeneous, then $\Omega$ is homogeneous.  

The domain $D$ is symmetric if it is homogeneous and admits
an involutive biholomorphism with an isolated (or, equivalently, a unique) fixed point. If $D$ is
symmetric, then $\Omega$ is symmetric, that is, homogeneous and
self-dual. Conversely, if $\Omega$ is symmetric \emph{and $D$ is a
tubular domain}, then  $D$ is symmetric (cf., e.g.,~\cite[Theorem]{Satake} for more details on various characterizations of symmetric Siegel domains).

The \v{S}ilov boundary  of $D$, that is, the smallest closed subset of $\overline D$ on which every bounded continuous function on $\overline D$ which is holomorphic on $D$ has the same supremum as on $\overline D$, is 
\[
bD \coloneqq \Set{ (\zeta, z) \in E \times F_\C \colon\rho(\zeta,z) =0 } ,
\]
and admits a natural $2$-step nilpotent Lie group structure whose product is best described under the identification $ b D\ni (\zeta,
x+i\Phi(\zeta,\zeta))\mapsto(\zeta,x)\in E\times F$. Namely,
\[
(\zeta,x)(\zeta',x')=(\zeta+\zeta',x+x'+2\Ima
\Phi(\zeta,\zeta')),
\]
for $(\zeta,x),(\zeta',x')\in E\times F$,  see e.g.~\cite[Section 1.1]{CalziPeloso}. We denote by $\Nc$ the set $E\times F$ endowed with this group structure. 

Notice that  $\rho$ maps $ D$ into $\Omega$, and that
the fibres $b D+(0, i h)$, $h\in \Omega$, of $\rho$  give rise to a foliation of $D$. Given a function $f$
defined on $D$, we shall often denote by $f_h$ its restriction
to $b D+(0,i h)$, interpreted as a function on $\Nc$ for the sake of
convenience. Explicitly,
\[
f_h (\zeta,x)=f(\zeta,x+i\Phi(\zeta)+i h)
\]
for every $h\in \Omega$ and for every $(\zeta,x)\in \Nc$. Note
that, identifying  $b D+(0,i h)$ with $\Nc$ as above for every $h\in
\Omega$, we get a left action of $b D$ on $D$ by affine
biholomorphisms.

For $p,q\in ]0,\infty]$ and $\vs\in\R^r$, the
weighted Bergman spaces are defined as\footnote{The definitions of the
rank $r$ of $\Omega$,  of the `generalized power functions'
$\Delta_\Omega^{\vect s}$ and of $\vd\in\R^r$   are deferred to Section~\ref{sec:2}.} 
\begin{equation}\label{eq:1}
A^{p,q}_\vs (D)\coloneqq \Set{  f\in\Hol(D)\colon\int_\Omega \Big(
\int_\Nc\abs{f_h(\zeta,x)}^p \,  \dd (\zeta,x) \Big)^{q/p}
\Delta^{q \vs}_\Omega (h)  \, \dd \nu_\Omega(h)<\infty\,  } 
\end{equation}
(modification if $\max(p,q)=\infty$), where $\dd (\zeta,x)$ denotes a
Haar measure on $\Nc$ and $\nu_\Omega$ denotes a positive
$G(\Omega)$-invariant measure  on $\Omega$, both  fixed and  unique
up to a multiplicative constant. We shall sometimes simply write $A^{p,q}_\vs$ instead of $A^{p,q}_\vs (D)$. 

We remark that the spaces $A^{p,p}_\vs$ are the more `classical' weighted Bergman spaces,  the unweighted case corresponding to the value $\vs=-\vd/p$, while the
spaces $A^{p,\infty}_{\vect 0}$ are the classical Hardy spaces.  
 
We also remark that the analysis  of  holomorphic
function spaces on homogeneous Siegel domains of Type
II is a quite active area of research, see, e.g.,~\cite{ACMPS,AMPS,MPS} and
references therein for some recent results.  

In this paper, we are interested in the study of Toeplitz and Cesàro-type operators on the weighted Bergman spaces $A^{p,q}_\vs(D)$.  

\medskip

We now briefly outline the structure of this paper.
In Section~\ref{sec:2}, we review some basic facts concerning
homogeneous cones and homogeneous Siegel domains. In
Section~\ref{sec:fnct-sp}, we review the definitions and the basic
properties of the function spaces on homogeneous Siegel domain that
are relevant to our analysis. 
In Section~\ref{sec:Schatten}, we recall some basic facts about
Schatten classes of operators between two hilbertian spaces. These
results should be known, but we have not been able to find them stated
in the generality we needed. Even though we provide no proofs, we
briefly indicate how the result present in the literature should be
adapted. 
In Section~\ref{sec:Toeplitz}, we present our main results about
Toeplitz operators between weighted Bergman spaces on homogeneous
Siegel domains. We both provide necessary conditions
(cf.~Proposition~\ref{prop:11}) and sufficient conditions
(cf.~Theorem~\ref{prop:14}) for continuity and compactness of Toeplitz
operators between the spaces  $A^{p,q}$.  In addition, we provide
necessary conditions and sufficient conditions on the symbol $\mi$
(cf.~Theorem~\ref{prop:33}) for a Toeplitz operator between the spaces
$A^{2,2}$ to belong to the Schatten class $\Lin^p$. 
In Section~\ref{sec:Cesaro}, we present our main results about
Cesàro-type operators between weighted Bergman spaces on homogeneous
Siegel domains. We both provide necessary conditions
(cf.~Proposition~\ref{prop:5bis} and~\ref{prop:9}) and sufficient
conditions (cf.~Proposition~\ref{prop:4}) for continuity and
compactness of Cesàro-type operators between the spaces $A^{p,q}$. In
addition, we characterize the Cesàro-type operators which induce
isomorphisms onto their image between the spaces $A^{p,p}$
(cf.~Proposition~\ref{prop:29}) and the Cesàro-type operators between
the spaces $A^{2,2}$ to belong to the Schatten class $\Lin^p$
(cf.~Theorem~\ref{prop:30}).

\section{Homogeneous Siegel domains of Type  II}
\label{sec:2}

We present here without proof some basic facts concerning homogeneous Siegel domains of type II. For a more detailed exposition, see~\cite{CalziPeloso}, of which we keep the notation as far as possible.

We shall denote by $E$ a complex hilbertian space of finite dimension $n$, and by $F$ a real hilbertian space of finite dimension $m$. We denote by $F_\C$ the complexification of $F$. Given an open convex cone $\Omega\subseteq F$ not containing any affine line, and an hermitian mapping $\Phi\colon E\times E\to F_\C$ such that
\[
\Phi(\zeta)\coloneqq \Phi(\zeta,\zeta)\in \overline \Omega \setminus \Set{0}
\]
for every non-zero $\zeta\in E$, we denote by $D$ associated the Siegel domain (of type II), that is,
\[
D\coloneqq \Set{(\zeta,z)\in E\times F_\C\colon \rho(\zeta,z)\coloneqq \Ima z-\Phi(\zeta)\in \Omega}.
\]
In order that $D$ be homogeneous (that is, in order that the biholomorphisms of $D$ act transitively) it is necessary and sufficient that there is a Lie group $T_+$ such that the following hold:
\begin{itemize}
	\itb $T_+$ acts linearly and simply transitively (on the left) on $\Omega$;
	
	\itb for every $t\in T$ there is $g\in GL(E)$ such that $t\cdot \Phi=\Phi(g\times g)$.
\end{itemize}
In this case, $T_+$ acts, by transposition, (on the right) on the dual cone 
\[
\Omega'\coloneqq \Set{\lambda\in F'\colon \forall h\in \overline{\Omega}\setminus \Set{0}\quad \langle \lambda, h\rangle >0}.
\] 

Observe that $T_+/[T_+,T_+]$ is canonically isomorphic to $(\R_+^*)^r$ for some $r\in \N$, which is the rank of $\Omega$ (and $D$). In order to avoid trivialities, we shall assume that $r>0$, that is, that $F\neq \Set{0}$. Notice that $\Omega=\R_+^*$ when $r=1$.
Once we fix an analytic surjective strict morphism $\Delta\colon T_+\to (\R_+^*)^r$, we may describe the characters of $T_+$ as the `generalized power functions' 
\[
\Delta^\vect s=\Delta_1^{s_1}\cdots \Delta_r^{s_r},
\]
for every $\vect s\in \C^r$. Once base-points $e_\Omega$ and $e_{\Omega'}$ are chosen in $\Omega$ and $\Omega'$, respectively, the characters $\Delta^{\vect s}$ can be transferred to generalized power functions $\Delta^{\vect s}_\Omega$ and $\Delta^{\vect s}_{\Omega'}$ on $\Omega$ and $\Omega'$, respectively. Precisely,
\[
\Delta^{\vect s}_\Omega(t\cdot e_\Omega)=\Delta^{\vect s}_{\Omega'}(e_{\Omega'}\cdot t)=\Delta^{\vect s}(t)
\]
for every $t\in T_+$.

As a matter of fact, it is possible to find $\Delta$ in such a way that the following hold (cf.~\cite[Sections 2.1--2.3 and 2.5]{CalziPeloso}):
\begin{itemize}
	\itb there is $\vd\in (\R_-^*)^r$ such that the measures
	\[
	\nu_\Omega\coloneqq \Delta_\Omega^\vd \cdot \Hc^m \qquad \text{and} \qquad \nu_{\Omega'}\coloneqq \Delta_{\Omega'}^\vd \cdot \Hc^m
	\]
	are invariant under the linear automorphisms of $\Omega$ and $\Omega'$, respectively;\footnote{Here, $\Hc^m$ denotes the Hausdorff measure on $\Omega$ and $\Omega'$, respectively.}
	
	\itb there are $\vm,\vm'\in \N^r$ such that the Laplace transform of $\Delta_\Omega^{\vect s}\cdot \nu_\Omega$ and $\Delta_{\Omega'}^{\vect s}\cdot \nu_{\Omega'}$ has a non-empty domain if and only if $\Rea \vs \in \vm+(\R_+^*)^r$ and $\Rea \vs \in \vm'+(\R_+^*)^r$, respectively;
	
	\itb for every $\Rea \vs \in \vm+(\R_+^*)^r$ (resp.\ $\Rea\vs \in \vm'+(\R_+^*)^r$) the Laplace transform of $\Delta_\Omega^{\vect s}\cdot \nu_\Omega$  (resp.\ $\Delta_{\Omega'}^{\vect s}\cdot \nu_{\Omega'}$) is defined on $\Omega'$ (resp.\ $\Omega$), and equals
	\[
	\Gamma_\Omega(\vs)\Delta_{\Omega'}^{-\vs} \qquad \text{(resp.\ }\Gamma_{\Omega'}(\vs)\Delta_{\Omega}^{-\vs}\text{),}
	\]
	thereon, where
	\[
	\Gamma_\Omega(\vs)=c \prod_{j=1}^r \Gamma\left(s_j-\frac{m_j}{2}\right) \qquad    \text{(resp.\ } \Gamma_{\Omega'}(\vs)=c \prod_{j=1}^r \Gamma\left(s_j-\frac{m'_j}{2}\right)\text{),}
	\]
	for a suitable constant $c>0$;
	
	\itb there is $\vb\in \R_-^r$ such that $\Delta^{\vb}(t)=\det_\R(g)$ for every $t\in T_+$ and for every $g\in GL(E)$ such that $t\cdot \Phi=\Phi(g\times g)$;
	
	\itb  $\vb=\vect{0}$ if and only if $E=\Set{0}$;
	
	\itb the measure $\nu_D\coloneqq (\Delta^{\vect b+2\vect d}\circ \rho) \cdot \Hc^{2 n+2m}$ on $D$ is invariant under all biholomorphisms of $D$.
\end{itemize}

Observe that, when $r=1$, the holomorphic family of measures $\big(\frac{1}{\Gamma(s)} (\,\cdot\,)^{s-1}\cdot \Hc^1\big)_{\Rea s>0}$  extends uniquely to a holomorphic family of tempered distributions on $\C$, giving rise to the so-called Riemann--Liouville operators. 
Also in the general case it is possible to prove that there is a unique holomorphic family  $(I^{\vect s}_\Omega)_{\vs\in \C^r}$ of tempered distributions on $F$  such that $I^{\vect s}_\Omega= \frac{1}{\Gamma_\Omega(\vs)}\Delta^{\vect s}_\Omega \cdot \nu_\Omega$ for $\Rea\vs\in \frac 1 2 \vect m+(\R_+^*)^r$ (cf.~\cite[Lemma 2.26, Definition 2.27, and Proposition 2.28]{CalziPeloso}). We therefore call `Riemann--Liouville operators' the operators of convolution by the distributions $I^{\vs}_\Omega$.

Notice that $\Delta^\vs_\Omega\circ \rho$ can be interpreted as the restriction to the diagonal of a `sesqui-holomorphic' function $B^\vs$ defined on $D\times D$ which is of particular importance in the study of weighted Bergman spaces. Explicitly,
\[
B^\vs_{(\zeta',z')}(\zeta,z)=\Delta_\Omega^\vs\left( \frac{z-\overline{z'}}{2 i}-\Phi(\zeta,\zeta') \right)
\]
for every $(\zeta,z),(\zeta',z')\in D$. Obviously, the same definition can be extended to the case in which one at most between $(\zeta,z)$ and $(\zeta',z')$ belongs to $\overline D$.

Observe that $\Nc\coloneqq E\times F$, endowed with the $2$-step nilpotent Lie group structure given by the product
\[
(\zeta,x)(\zeta',x')\coloneqq (\zeta+\zeta', x+x'+2 \Ima \Phi(\zeta,\zeta'))
\]
for $(\zeta,x),(\zeta',x')\in E\times F$, acts freely and affinely on the complex space $E\times F_\C$ as follows:
\[
(\zeta,x)\cdot (\zeta',z')\coloneqq (\zeta+\zeta',x+i\Phi(\zeta)+z'+2 i\Ima\Phi(\zeta',\zeta))
\]
In particular, $\Nc$ acts simply transitively on the \v Silov boundary $b D=\Nc\cdot \vect{0}$ of $D$, with which it can therefore be identified, and induces an action on $D$.

Fourier analysis on $\Nc$ plays a relevant role in the study of various function spaces of holomorphic functions on $D$. For our purposes, a detailed presentation of the representations of $\Nc$ is superfluous, and we shall only present some basic facts. 
Observe first that, for every $\lambda\in F'$, the group $\Nc/\ker \lambda$ is isomorphic to the product of a Heisenberg group and an abelian group; it is actually isomorphic to a Heisenberg group if $\lambda$ is in the complement of a proper algebraic set. Therefore, the Stone--von Neumann theorem (cf., e.g.,~\cite[Theorem 1.50]{Folland}) shows that there is (up to unitary equivalence) a unique irreducible unitary representation $\pi_\lambda$ of $\Nc$ is some hilbertian space $H_\lambda$ such that $\pi_\lambda(0,i x)=e^{-i \langle \lambda, x\rangle} I$ for every $x\in F$.
It turns out that these representations are sufficient to get a Plancherel formula. More explicitly, (cf.~\cite[Corollary 1.17 and Proposition 2.30]{CalziPeloso})
\[
\norm{f}_{L^2(\Nc)}^2= c\int_{F'} \norm{\pi_\lambda (f)}_{\Lin^2(H_\lambda)}^2 \Delta_{\Omega'}^{-\vb}(\lambda)\,\dd \lambda
\]
for every $f\in L^1(\Nc)\cap L^2(\Nc)$, where $\Lin^2(H_\lambda)$ denotes the space of Hilbert--Schmidt endomorphisms of $H_\lambda$.
Notice that $\Delta_{\Omega'}^{-\vb}$ is actually a polynomial (so that it is defined on the whole of $F'$), cf.~\cite[Proposition 2.30]{CalziPeloso}.

Since $\pi_\lambda(f_h)=0$ for almost every $\lambda\not \in \Omega'$, for every $h\in \Omega$, and for every $f$ in the space $ A^{p,q}_\vs$ (to be defined below), $p\in ]0,2]$ (cf.~\cite[Corollary 1.37 and 3.3, and Proposition 3.2]{CalziPeloso}),  we shall only describe $\pi_\lambda$ for $\lambda\in \Omega'$ (`Bargmann representation'). In this case, $H_\lambda\coloneqq \Hol(E)\cap L^2(\nu_\lambda)$, with $\nu_\lambda= \ee^{-2\langle\lambda, \Phi(\,\cdot\,)\rangle}\cdot \Hc^{2 n}$, and  
\begin{equation}\label{eq:pi-lambda}
	\pi_\lambda(\zeta,x)  \psi  (\omega)
	\coloneqq \ee^{\langle \lambda_\C, -i x +2 \Phi(\omega, \zeta)-\Phi(\zeta)\rangle}  \psi (\omega- \zeta),
\end{equation} 
for every $\psi \in H_\lambda$, for every $\omega\in E$, and for every $(\zeta,x)\in \Nc$. 
In addition, if $P_{\lambda,0}$ denotes the self-adjoint projector of $H_\lambda$ onto the space of constant functions, then $\pi_\lambda(f_h)=\pi_\lambda(f_h)P_{\lambda,0}$ for almost every $\lambda\in \Omega'$, for every $h\in \Omega$, and for every $f\in A^{p,q}_\vs$, $p\in]0,2]$ (cf.~\cite[Proposition 1.19 and 3.2, and Corollary 3.3]{CalziPeloso}).

We conclude this section with some remarks concerning lattices. We first endow $D$ with the Bergman metric, which is the complete K\"ahler metric defined by
\[
\partial_v \overline{\partial_w} \log (\Delta^{\vect b+2 \vect d}\circ \rho)(\zeta,z)
\]
for every $(\zeta,z)\in D$ and for every $v,w\in E\times F_\C$ (cf.~\cite[Section 2.5]{CalziPeloso}). We denote by $B((\zeta,z),R)$ the corresponding open ball of centre $(\zeta,z)$ and radius $R$.

We endow $\Omega$ with the quotient metric induced by the submersion $\rho\colon D\to \Omega$, and $\Omega'$ with the Riemannian metric induced by the correspondence $\Omega\ni t\cdot e_\Omega \mapsto e_{\Omega'}\cdot t\in \Omega'$. We denote by $B_\Omega(h,R)$ and $B_{\Omega'}(\lambda,R)$ the corresponding open balls of centre $h$ and $\lambda$, respectively, and radius $R$. 

A $(\delta,R)$-lattice on $\Omega$, with $\delta>0$ and $R>1$, is a family $(h_k)_{k\in K}$ of elements of $\Omega$ such that the  $B_\Omega(h_k,\delta)$ are pairwise disjoint, while the   $B_\Omega(h_k,R\delta)$ cover $\Omega$. 
Observe that any maximal $(2\delta)$-separated family of elements of $\Omega$ is a $(\delta,2)$-lattice, so that $(\delta,2)$-lattices always exist.
Lattices on $\Omega'$ are defined similarly.

In order to define lattices on $D$, though, we need to be a little more cautious. Since we wish to deal with mixed-norm spaces, it is more convenient to consider a two-parameter family $(\zeta_{j,k},z_{j,k})_{j\in J,k\in K}$ of elements of $D$ such that the $B((\zeta_{j,k},z_{j,k}),\delta)$ are pairwise disjoint, the $B((\zeta_{j,k},z_{j,k}),R\delta)$ cover $D$ (as for usual lattices), and such that there is a $(\delta,R)$-lattice $(h_k)_{k\in K}$ on $\Omega$ such that $\rho(\zeta_{j,k},z_{j,k})=h_k$ for every $j\in J$ and for every $k\in K$.
Refining the argument which gives rise to $(\delta,2)$-lattices on $\Omega$ and $\Omega'$, it is possible to prove that there are $(\delta,4)$-lattices on $D$ for every fixed $\delta>0$ (cf.~\cite[Lemma 2.55]{CalziPeloso}).

\section{Function spaces}\label{sec:fnct-sp}

In this section we define the main function spaces we shall consider, and state some of their properties. We refer the reader to~\cite{CalziPeloso}  for a more thorough exposition.

\begin{deff}\label{def:Lpqs}
	Take $\vs\in \R^r$ and $p,q\in ]0,\infty]$, and define 
	\[
	L^{p,q}_\vs(D)\coloneqq\Set{ f\colon D\to \C\colon  f \text{ is measurable}, \int_\Omega  \Big(    \Delta_\Omega^{\vs}(h) \norm{f_h}_{L^p(\Nc)}  \Big) ^q\,\dd  \nu_\Omega(h)<\infty  }  
	\]
	(modification when $q=\infty$). We define  $L^{p,q}_{\vs,0}(D)$ as the
	closure of $C_c( D)$ in $L^{p,q}_{\vs}(D)$, and set
	\[
	A^{p,q}_\vs(D)= L^{p,q}_\vs(D)\cap \Hol(D), \qquad\text{and}\qquad
	A^{p,q}_{\vs,0}(D)
	= L^{p,q}_{\vs,0}(D) \cap \Hol(D).
	\]
\end{deff}

Notice that $L^{p,q}_{\vs,0}(D)=L^{p,q}_{\vs}(D)$  and $A^{p,q}_{\vs,0}(D)=A^{p,q}_{\vs}(D)$ if (and only if) $p,q<\infty$.
In addition, it is not hard to prove that $A^{p,q}_{\vs,0}(D)\neq \Set{0}$ (resp.\   $A^{p,q}_{\vs}(D)\neq \Set{0}$) if and only if  $\vs\in \frac{1}{2 q}\vect m+(\R_+^*)^r$ (resp.  $\vs\in \R_+^r$ if $q=\infty$), cf.~\cite[Proposition~3.5]{CalziPeloso}. 

Observe that the spaces $A^{p,q}_\vs(D)$ and $A^{p,q}_{\vs,0}(D)$ are complete metrizable topological vector spaces (Banach spaces when $p,q\Meg 1$) and embeds continuously into $\Hol(D)$, endowed with the topology of compact convergence. In particular, $A^{2,2}_\vs(D)$ is a reproducing kernel hilbertian space, with reproducing kernel
\[
K_\vs ((\zeta,z),(\zeta',z'))
\coloneqq c\frac{\Gamma_{\Omega'}(2 \vs-\vb-\vd)}{\Gamma_\Omega(2\vs)}
B_{(\zeta',z')}^{\vb+\vd-2\vs} (\zeta,z)
\]
for a suitable constant $c>0$, cf.~\cite[Remark 3.12]{CalziPeloso}. 
For $\vs\in \vb+\vd-\frac 1 2 \vm-(\R_+^*)^r$ we may therefore consider the weighted Bergman projector
\[
P_\vs\colon f \mapsto c_\vs \int_D f(\zeta',z')B_{(\zeta',z')}^{\vs} \Delta^{-\vs}(\rho(\zeta',z')) \,\dd \nu_D(\zeta,z),
\]
where
\[
c_{\vs}\coloneqq c\frac{\Gamma_{\Omega'}(-\vs)}{\Gamma_\Omega(\vb+\vd-\vs)}.
\]
Analogously, the Hardy space $A^{2,\infty}_{\vect0}(D)$ is a reproducing kernel hilbertian space, and its reproducing kernel (the `Cauchy--Szeg\H o kernel') is given by
\[
S_{(\zeta,z)}(\zeta',x') \coloneqq
c' \big( B^{\vb+\vd}_{(\zeta,z)}\big)_0(\zeta',x')
\]
for a suitable constant $c'>0$, cf.~\cite[Lemma 5.1]{CalziPeloso}.\footnote{Since the boundary value mapping 
	\[
	f\mapsto f_0\coloneqq\lim_{h\to 0} f_h
	\]
	is an isometry from the Hardy space $A^{2,\infty}_{\vect 0}(D)$ onto a closed subspace of $L^2(\Nc)$, it is customary to write the Cauchy--Szeg\H o kernel as a function on $\Nc\times D$  instead of as a function on $D\times D$ as one may expect.  }  One may therefore reconstruct every $f\in A^{2,\infty}_{\vect 0}(D)$ from their boundary values $f_0\coloneqq \lim_{h\to 0} f_h$. More precisely,
\[
f(\zeta,z)=\langle f_0| S_{(\zeta,z)}\rangle_{L^2(\Nc)} 
\]
for every $(\zeta,z)\in D$.

This reconstruction formula is crucial for the study of boundary values of the spaces $A^{p,q}_\vs(D)$. 
Let us first describe the Besov spaces of analytic type (on $\Nc$) which are the `natural' candidates for the boundary value spaces of the weighted Bergman spaces considered above.
Since the non-commutative Fourier transform of tempered distributions on $\Nc$ is not easy to manage, we shall first introduce some spaces of test functions which are particularly well-suited to our analysis.
We first define
\[
\Sc_\Omega(\Nc)\coloneqq \Set{\psi\in \Sc(\Nc)\colon  \exists \phi\in C^\infty_c(\Omega') \;\forall \lambda\in \Omega'\;\; \pi_\lambda(\psi)= \phi(\lambda) P_{\lambda,0}, \text{ while } \pi_\lambda(\psi)=0 \text{ for a.e.\ $\lambda\not \in \Omega'$}},
\]
and then $\Sc_{\Omega,L}(\Nc)*\Sc_\Omega(\Nc)$, endowed with the inductive limit of the
topologies induced by $\Sc(\Nc)$ on its subspaces $\Sc(\Nc)*\psi$,
$\psi\in \Sc_{\Omega}(\Nc)$.\footnote{It is not hard to see that this  definition is equivalent to~\cite[Definition 4.4]{CalziPeloso}.}
We denote by  $\Sc_{\Omega,L}'(\Nc)$ the dual of $\Sc_{\Omega,L}(\Nc)$.
See~\cite[Propositions 4.2 and 4.5, and Lemma~4.14]{CalziPeloso} for a proof of the following result.

\begin{prop}\label{prop:4.2}
The following hold:
\begin{enumerate}
\item[\em(1)] the mapping $\Fc_\Nc\colon \varphi\mapsto [\lambda
\mapsto \tr(\pi_\lambda(\varphi) ) ] $ induces an isomorphism of
$\Sc_\Omega(\Nc)$ onto $C^\infty_c(\Omega')$; 

\item[\em(2)]  given two $(\delta,R)$-lattices $(\lambda_k)_{k\in K}$ and $(\lambda'_{k'})_{k'\in K'}$ on $\Omega'$, and two families $(\psi_k)_{k\in K}, (\psi'_{k'})_{k'\in K'}$ of elements of $\Sc_\Omega(\Nc)$ such that $((\Fc_\Nc \psi_k)(\,\cdot\, t_k))$ and $((\Fc_\Nc \psi'_{k'})(\,\cdot\, t'_{k'}))$ are bounded families of positive elements of $C^\infty_c(\Omega')$, where $t_k,t'_{k'}\in T_+$ are such that $\lambda_k=e_{\Omega'}\cdot t_k$ and $\lambda'_{k'}=e_{\Omega'}\cdot t'_{k'}$, and
\[
\sum_k \Fc_\Nc\psi_k, \sum_{k'} \Fc_\Nc \psi_{k'}\Meg 1
\]
on $\Omega'$,  one has  
\[
\Big\| \Delta_{\Omega'}^{\vs}(\lambda_{k'}) \big\|u*\psi'_{k'}
\big\|_{L^p(\Nc)}  \Big\|_{\ell^q(K')}
\approx\Big\| \Delta_{\Omega'}^{\vs}(\lambda_k) \big\| u*\psi_k
\big\|_{L^p(\Nc)}  \Big\|_{\ell^q(K)},
\]
for every $u\in \Sc_{\Omega,L}'(\Nc)$.  
	\end{enumerate} 
\end{prop}

\begin{deff}
Take $\vs\in \R^r$, $p,q\in ]0,\infty]$. Given
$(\lambda_k)_{k\in K}$ and $(\psi_k)$ as in
Proposition~\ref{prop:4.2}, we define $B^{\vs}_{p,q}(\Nc,\Omega)$ as the space of $u\in \Sc_{\Omega,L}'(\Nc)$ such that
\[
(\Delta^{\vs}_{\Omega'}(\lambda_k) (u*\psi_k))_k\in  \ell^q (K;L^p(\Nc)),
\]
endowed with the corresponding topology. We denote by  $\mathring{B}^{\vs}_{p,q}(\Nc,\Omega)$ the closure of (the canonical image of) $\Sc_{\Omega,L}(\Nc,\Omega)$ in $B^{\vs}_{p,q}(\Nc,\Omega)$, which can be described as the space of the $u\in\Sc_{\Omega,L}'(\Nc)$ such that  
\[
(\Delta^{\vs}_{\Omega'}(\lambda_k) (u*\psi_k))_k\in \ell^q_0(K; L^p_0(\Nc)).
\]
(cf,~\cite[Theorem 4.23]{CalziPeloso}).
\end{deff}

See~\cite[Proposition 4.20 and Theorem 4.23]{CalziPeloso} for a proof
of the following result.  Here and in what follows, we put
$p'\coloneqq \max(1,p)'$ for every $p\in ]0,\infty]$, so that
$p'=\infty$ if $p\meg 1$, and $\frac 1 p+\frac{1}{p'}=1$ if $p\Meg
1$.  

\begin{prop}
Take $p,q\in ]0,\infty]$ and $\vs\in \R^r$. Then,  the canonical sesquilinear pairings on $\Sc_{\Omega,L}(\Nc)\times \Sc_{\Omega,L}'(\Nc)$  and on $\Sc'_{\Omega,L}(\Nc)\times \Sc_{\Omega,L}(\Nc)$ induce  unique continuous sesquilinear pairings
\[
\langle\,\cdot\,\vert \,\cdot\,\rangle\colon \mathring B^{\vs}_{p,q}(\Nc,\Omega)\times B^{-\vs-(1/p-1)_+(\vect b+\vect d)}_{p',q'}(\Nc,\Omega)\to \C,
\]
and
\[
\langle\,\cdot\,\vert \,\cdot\,\rangle\colon  B^{\vs}_{p,q}(\Nc,\Omega)\times \mathring B^{-\vs-(1/p-1)_+(\vect b+\vect d)}_{p',q'}(\Nc,\Omega)\to \C,
\]
respectively.
\end{prop}

We may now introduce an extension operator from some of the spaces $B^{p,q}_{-\vs}(\Nc,\Omega)$ into suitable weighted Bergman spaces.

\begin{deff}
Take $p,q\in ]0,\infty]$ and $\vs\in \frac 1 p (\vb+\vd)+\frac{1}{2 q'}\vm'+(\R_+^*)^r$, and observe that $S_{(\zeta,z)}\in \mathring B^{\vs-(1/p-1)_+(\vb+\vd)}_{p',q'}(\Nc,\Omega)$ by~\cite[Lemma 5.1]{CalziPeloso}. Define a continuous linear mapping $\Ec\colon B^{p,q}_{-\vs}(\Nc,\Omega)\to A^{\infty,\infty}_{\vs-(\vb+\vd)/p}(D)$ by
\[
\Ec u(\zeta,z)\coloneqq \langle u  \vert S_{(\zeta,z)}\rangle
\]
for every $(\zeta,z)\in D$. Define
\[
\widetilde A^{p,q}_\vs(D)\coloneqq \Ec(B^{-\vs}_{p,q}(\Nc,\Omega)) \qquad \text{and} \qquad \widetilde A^{p,q}_{\vs,0}(D)\coloneqq \Ec(\mathring B^{-\vs}_{p,q}(\Nc,\Omega)),
\]
endowed with the corresponding (image) topology.
\end{deff}

Cf.~\cite[Theorem 5.2, Proposition 5.4 and its proof, and Corollary 5.11]{CalziPeloso} for a proof of the following result.

\begin{prop}\label{prop:23bis}
Take $p,q\in ]0,\infty]$, and $\vs\in 
\frac{1}{p}(\vb+\vd)+\frac{1}{2 q'}\vect{m'}+(\R_+^*)^r $. 
Then, the following hold:
\begin{enumerate}
\item[\em(1)] $(\Ec u)_h$ converges to $u$ in
$B^{-\vs}_{p,q}(\Nc,\Omega)$ (resp.\ in $\Sc'_{\Omega,L}(\Nc)$)
for everu $u\in \mathring B^{-\vs}_{p,q}(\Nc,\Omega)$ (resp.\ for
every $u\in B^{-\vs}_{p,q}(\Nc,\Omega)$);  

\item[\em(2)] if, in addition, $\vs\in \frac{1}{2 q}\vect m+(\R_+^*)^r $, then there are continuous inclusions
\[
\Ec(\Sc_{\Omega,L}(\Nc))\subseteq A^{p,q}_{\vs}(D) \subseteq
\widetilde A^{p,q}_{\vs}(D)
\qquad
(\text{resp.\ } \Ec(\Sc_{\Omega,L}(\Nc))\subseteq A^{p,q}_{\vs,0}(D) \subseteq \widetilde A^{p,q}_{\vs,0}(D))  ;
\]

\item[\em(3)] if, further,$\vs\in \frac{1}{2 q}\vect m+\big( \frac{1}{2 \min(p,p')}-\frac{1}{2q} \big)_+\vect{m'}  +(\R_+^*)^r,$
then,
\[
A^{p,q}_{\vs}(D)=\widetilde A^{p,q}_{\vs}(D)  \qquad\text{and} \qquad
A^{p,q}_{\vs,0}(D)=\widetilde A^{p,q}_{\vs,0}(D).
\]
\end{enumerate}
\end{prop}

We can now present a sufficient condition for the continuity of the projectors $P_\vs$. See~\cite[Proposition 5.21 and Corollary 5.26]{CalziPeloso} for a proof of the following result.

\begin{prop}\label{prop:26}
Take $p,q\in [1,\infty]$ and $\vs,\vs'\in \R^r$ such that the following hold:
\begin{itemize}
\itb $\vs\in \sup \left( \frac{1}{2 q} \vm, \frac 1 p (\vb+\vd)+\frac{1}{2 q'}\vm' \right)+(\R_+^*)^r$;

\itb $\vs+\vs'\in \inf\left( \vb+\vd-\frac{1}{2 q'}\vm, \frac{1}{p}(\vb+\vd)-\frac{1}{2q}\vm' \right)-(\R_+^*)^r$;

\itb $A^{p,q}_{\vs,0}(D)=\widetilde A^{p,q}_{\vs,0}(D)$;

\itb $A^{p,q}_{\vb+\vd-\vs-\vs',0}(D)=\widetilde A^{p,q}_{\vb+\vd-\vs-\vs',0}(D)$.
\end{itemize}
Then, $P_{\vs'}$ induces a continuous linear projector of $L^{p,q}_{\vs,0}(D)$ onto $A^{p,q}_{\vs,0}(D)$.
\end{prop}

We conclude this section with some remarks on the atomic decomposition of the spaces $A^{p,q}_\vs(D)$, and its connection with the spaces $\widetilde A^{p,q}_\vs(D)$.

\begin{deff}\label{def:L-property}
We say that property $\atomic^{p,q}_{\vs,\vect{s'},0}$ (resp.\
$\atomic^{p,q}_{\vs,\vect{s'}}$) holds if for every $\delta_0>0$
there is a $(\delta,4)$-lattice $(\zeta_{j,k},z_{j,k})_{j\in J, k\in K}$, with $\delta\in ]0,\delta_0]$,  such that, defining $h_k\coloneqq \rho(\zeta_{j,k},z_{j,k})$ for every $k\in K$ andfor some (hence every) $j\in J$, the mapping 
\[
\Psi\colon\lambda \mapsto \sum_{j,k} \lambda_{j,k}
B_{(\zeta_{j,k},z_{j,k})}^{\vect{s'}} \Delta_\Omega^{(\vb+\vect
d)/p-\vs-\vect{s'}}(h_k) 
\]
is well defined (with locally uniform convergence of the sum) and maps
$\ell^{p,q}_0(J,K) $ into $A^{p,q}_{\vs,0}(D)$ continuously
(resp.\ maps $\ell^{p,q}(J,K) $ into $A^{p,q}_{\vs}(D)$
continuously).\footnote{Define
\[
\ell^{p,q}(J,K)\coloneqq \Set{\lambda\in \C^{J\times K}\colon ((\lambda_{j,k})_{j\in J})_{k\in K}\in \ell^q(K;\ell^p(J))},
\]
endowed with the corresponding quasi-norm, and define
$\ell^{p,q}_0(J,K)$ as the closure of $\C^{(J\times K)}$ in
$\ell^{p,q}(J,K)$.} 
We say that property $\atomics^{p,q}_{\vect s,\vect{s'},0}$ (resp.\ $\atomics^{p,q}_{\vs,\vect{s'}}$) holds if, for every $\delta_0>0$ as above, we may take $(\zeta_{j,k},z_{j,k})_{j\in J, k\in K}$ in such a way that the corresponding mapping $\Psi$ is onto.  
\end{deff}

See~\cite[Corollaries 5.14 and 5.16]{CalziPeloso} for a proof of the following result.

\begin{prop}\label{prop:25}
Take $p,q\in ]0,\infty]$, $\vs\in \sup\Big( \frac{1}{2 q}\vect m,
\frac{1}{p}(\vb+\vd)+\frac{1}{2 q'}\vect{m'}\Big)+(\R_+^*)^r $. Then, the following hold:
\begin{itemize}
\itb if $A^{p,q}_{\vect s,0}(D)=\widetilde A^{p,q}_{\vs,0}(D)$
(resp.\ $A^{p,q}_{\vect s}(D)=\widetilde A^{p,q}_{\vs}(D)$), then property $\atomics^{p,q}_{\vs, \vs',0}$ (resp.\ $\atomics^{p,q}_{\vs,\vs'}$) holds for every $\vs'\in \frac{1}{\min(1,p)}(\vect b+\vect d)-\frac{1}{2 q}\vect{m'} -
\left(\frac{1}{2 \min(1,p)}-\frac{1}{2q}  \right)_+\vect m -
\vs-(\R_+^*)^r$;

\itb if property $\atomic^{p,q}_{\vs, \vs',0}$ (resp.\ $\atomic^{p,q}_{\vs,\vs'}$) holds for every $\vs'$ in a translate of $-\R_+^r$, then $A^{p,q}_{\vect s,0}(D)=\widetilde A^{p,q}_{\vs,0}(D)$
(resp.\ $A^{p,q}_{\vect s}(D)=\widetilde A^{p,q}_{\vs}(D)$).
\end{itemize}
\end{prop}

\section{Toeplitz Operators}\label{sec:Toeplitz}

In this section, we study Toeplitz operators between the spaces
$A^{p,q}_\vs(D)$. We first provide some necessary conditions for
continuity and compactness of Toeplitz operators
(cf.~Proposition~\ref{prop:11}), and then add some corresponding
sufficient conditions (cf.~Theorem~\ref{prop:14}). As often happens
with this kind of operators, the two conditions only match when the measure
$\mi$ symbol of the given Toeplitz operator  $T_\mi$
 is positive. We conclude this section providing necessary
conditions and sufficient conditions for a Toeplitz operator to belong
to some Schatten class $\Lin^p(A^{2,2}_{\vs}(D); A^{2,2}_{\vs'}(D))$
(cf.~Theorem~\ref{prop:33}). 

\begin{deff}
We denote by $\cM(D)$ the space of Radon measures on $D$, and by
$\cM_+(D)$ the space of positive Radon measures on $D$. Given $\mi\in
\cM(D)$ and $R>0$, we define 
\[
M_R(\mi)\colon D\ni (\zeta,z)\ni \mapsto \abs{\mi}(B((\zeta,z),R))\in \C.
\]
\end{deff}

\begin{deff}
Take $\mi\in \cM(D)$ and $\vect{s'}\in \R^r$. Define
\[
T_{\mi,\vect{s'}} f\coloneqq \int_D B^{\vect{s'}}_{(\zeta,z)} f(\zeta,z)\,\dd \mi(\zeta,z)
\]
for every $\mi$-measurable function $f$ such that $B^{\vect{s'}}_{\,\cdot\,}(\zeta',z')f\in L^1(\mi)$ for every $(\zeta',z')\in D$. 
\end{deff}

Observe that $B^{\vect{s'}}_{\,\cdot\,}(\zeta',z')=\overline{B^{\vect{s'}}_{(\zeta',z')}}$ for every $(\zeta',z')\in D$, so that $T_{\mi,\vect{s'}} f$ is defined if and only if $B^{\vect{s'}}_{(\zeta',z')} f\in L^1(\mi)$ for every $(\zeta',z')\in D$.

\begin{lem}\label{lem:2}
Take $\vect{s'}\in \R^r$ and $\mi\in \cM(D)$. Let $f$ be a
$\mi$-measurable function on $D$, and assume that
$B^{\vect{s'}}_{(\zeta,z)}f\in L^1(\mi)$ for some $(\zeta,z)\in
D$. Then, $T_{\mi,\vect{s'}}f$ is a well-defined holomorphic function
on $D$. 
\end{lem}

\begin{proof}
Observe that, by~\cite[Theorem 2.47]{CalziPeloso}, there is a constant $C>0$ such that
\[
\frac{1}{C} \abs*{B_{(\zeta,z)}^{\vect{s'}}(\zeta',z')}
\meg \abs*{B_{(\zeta,z)}^{\vect{s'}}(\zeta'',z'')}\meg C\abs*{B_{(\zeta,z)}^{\vect{s'}}(\zeta',z')}
\]
for every $(\zeta,z),(\zeta',z'),(\zeta'',z'')\in D$ such that
$d((\zeta',z'),(\zeta'',z''))\meg 1$. Therefore,
$\abs{B^{\vect{s'}}_{\,\cdot\,}(\zeta',z')  f}\meg
C_3\abs{B^{\vect{s'}}_{(\zeta,z)}  f}$ on $D$, for every
$(\zeta,z),(\zeta',z')\in D$ such that $d((\zeta,z),(\zeta',z'))\meg
1$. Hence, the assumption shows that
$B^{\vect{s'}}_{\,\cdot\,}(\zeta,z)  f\in L^1(\mi)$ for every
$(\zeta,z)\in D$, so that $T_{\mi,\vect{s'}}f$ is well defined. Then,
Morera's theorem readily implies that $T_{\mi, \vect{s'}}f $ is
holomorphic. 
\end{proof}

\begin{prop}\label{prop:11}
Take $p,q\in ]0,\infty]$ and $\vect s,\vect{s'},\vect{s''}\in \R^r$ such that $\vect s\in \frac{1}{2 q}\vect m+(\R_+^*)^r$ (resp.\ $\vect s\in \R_+^r$ if $q=\infty$) and such that  $T_{\mi,\vect{s'}}$ induces a continuous  linear mapping of $A^{p,q}_{\vect s,0}(D)$ (resp.\ $A^{p,q}_{\vect s}(D)$) into $A^{p,q}_{\vect b+\vect d-\vect{s'}-\vect{s''}}(D)$.
Then, the following hold:
\begin{enumerate}
\item[\emph{(i)}] assume that $M_1 (\mi)\in L^{\infty,\infty}_{\vect b+\vect d-\vect{s'}-\vect{s''}}(D)$, and that $B^{\vect{s'}}_{(\zeta,z)}\in A^{p',q'}_{\vect{s''}+(1/p-1)_+(\vect b+\vect d)}(D)$ for some (hence every) $(\zeta,z)\in D$. If we denote by $V$ the closed vector subspace of $A^{p',q'}_{\vect{s''}+(1/p-1)_+(\vect b+\vect d)}(D)$ generated by the $B^{\vect{s'}}_{(\zeta,z)}$, as $(\zeta,z)$ runs through $D$, then
\[
\int_D f \overline g\,\dd \mi= c_{\vect{s'}}\int_D (T_{\mi,\vect{s'}}f )\overline g (\Delta_\Omega^{-\vect{s'}}\circ \rho)\,\dd \nu_D
\]
for every $f\in A^{p,q}_{\vect s,0}(D)$ (resp.\ $f\in A^{p,q}_{\vect s}(D)$), and for every $g\in V$;

\item[\emph{(ii)}] if $\mi$ is positive, then $M_1 (\mi)\in L^{\infty,\infty}_{\vect b+\vect d-\vect{s}-\vect{s''}}(D)$;

\item[\emph{(iii)}] if $\mi$ is positive, $p,q>1$, $\vect s\in  \frac{1}{p}(\vect b+\vect d)+\frac{1}{2 q'}\vect{m'} +(\R_+^*)^r$, $A^{p,q}_{\vect s,0}(D)=\widetilde A^{p,q}_{\vect s,0}(D)$, and  $T_{\mi,\vect{s'}}$ induces a compact linear mapping of $A^{p,q}_{\vect s,0}(D)$ into $A^{p,q}_{\vect b+\vect d-\vect{s'}-\vect{s''}}(D)$, then $M_1 (\mi)\in L^{\infty,\infty}_{\vect b+\vect d-\vect{s}-\vect{s''},0}(D)$.
\end{enumerate}
\end{prop}

This extends one implication of~\cite[Lemma 4.1]{NanaSehba}, where the case in which $\vect s=\vect{s''}\in \R \vect 1_r$, $\vect{s'}=\vect d-2 \vect s$, $p=q=2$, and $D$ is an irreducible symmetric tube domain, is considered.

We observe that, if $p,q>1$ and $A^{p',q'}_{\vect{s''}}(D)=\widetilde A^{p',q'}_{\vect{s''}}(D)$, then $V$ is simply $A^{p',q'}_{\vect{s''}}(D)$, thanks to~\cite[Corollary 5.14]{CalziPeloso}. We do not know if $V=A^{p',q'}_{\vect{s''}}(D)$ under the sole assumption that $p,q>1$.

Before we pass to the proof, we need a  lemma.

\begin{lem}\label{lem:8}
Take  $p,q\in ]0,\infty]$ and $\vect s,\vect{s'},\vect{s''}\in \R^r$ such that  $T_{\mi,\vect{s'}}$ induces a continuous  linear mapping of $A^{p,q}_{\vect s,0}(D)$ (resp.\ $A^{p,q}_{\vect s}(D)$) into $A^{p,q}_{\vect b+\vect d-\vect{s'}-\vect{s''}}(D)$. In addition, take $\vect{s'''}\in \N_{\Omega'}$. Then, $T_{\mi, \vect{s'}-\vect{s'''}}$ induces a continuous  linear mapping of $A^{p,q}_{\vect s,0}(D)$ (resp.\ $A^{p,q}_{\vect s}(D)$) into $A^{p,q}_{\vect b+\vect d-\vect{s'}-\vect{s''}+\vect{s'''}}(D)$, and
\[
(T_{\mi,\vect{s'}} f)*I^{-\vect{s'''}}_\Omega= \left( \vect{s'}+\frac 1 2 \vect {m'}  \right)_{\vect{s'''}} T_{\mi,\vect{s'}-\vect{s'''}} f
\]
for every $f\in A^{p,q}_{\vect s,0}(D)$ (resp.\ $f\in A^{p,q}_{\vect s}(D)$).
\end{lem}

\begin{proof}
Take $f\in A^{p,q}_{\vect s,0}(D)$ (resp.\ $f\in A^{p,q}_{\vect s}(D)$). By~\cite[Proposition 2.29 and Corollary 3.27]{CalziPeloso}, it will suffice to prove that, for every $k\in \N$,
\[
D^k (T_{\mi,\vect{s'}} f)(\zeta,z)= \int_D f(\zeta',z') D^k B_{(\zeta',z')}(\zeta,z)\,\dd \mi(\zeta',z') 
\]
for every $(\zeta,z)\in D$, where $D^k$ denotes the differential of order $k$. Observe that, by Cauchy's estimates and~\cite[Theorem 2.47]{CalziPeloso}, there are two constants $C_{k,(\zeta,z)},C'_{k,(\zeta,z)}>0$ such that
\[
\abs{D^j B_{(\zeta',z')}(\zeta'',z'')}\meg C_{k,(\zeta,z)}\max_{\overline B((\zeta'',z''),1)} \abs{ B_{(\zeta',z')}}\meg C'_{k,(\zeta,z)}\abs{B_{(\zeta',z')}(\zeta,z)}
\]
for every $(\zeta',z')\in D$, for every $(\zeta'',z'')\in B((\zeta,z),1)$, and for every $j=0,\dots, k$. Therefore, the assertion follows by induction on $k$ and the theorems of differentiation under the integral sign.
\end{proof}

\begin{proof}[Proof of Proposition~\ref{prop:11}.]
(i) Observe first that, if $f\in A^{p,q}_{\vect s}(D)$ and $g\in A^{p',q'}_{\vect{s''}+(1/p-1)_+(\vect b+\vect d)}(D)$, then $f g\in A^{1,1}_{\vect s+\vect{s''}}(D)\subseteq L^1(\mi)$ (cf.~\cite[Proposition 3.2]{CalziPeloso} and~\cite[Theorem 5.4]{CPCarleson}), so that $f\overline g\in L^1(\mi)$. Therefore, the sesquilinear form
\[
(f,g)\mapsto \int_D f \overline g\,\dd \mi- c_{\vect{s'}}\int_D T_{\mi,\vect{s'}}f \overline g (\Delta_\Omega^{-\vect{s'}}\circ \rho)\,\dd \nu_D
\]
is continuous on   $A^{p,q}_{\vect s}(D)\times A^{p',q'}_{\vect{s''}+(1/p-1)_+(\vect b+\vect d)}(D)$.
Then, observe that
\[
\int_D f \overline{B^{\vect{s'}}_{(\zeta,z)}}\,\dd \mi= T_{\mi,\vect{s'}} f= c_{\vect{s'}}\int_D T_{\mi,\vect{s'}}f \overline {B^{\vect{s'}}_{(\zeta,z)}} (\Delta_\Omega^{-\vect{s'}}\circ \rho)\,\dd \nu_D
\]
for every $f\in A^{p,q}_{\vect s,0}(D)$ (resp.\ $f\in A^{p,q}_{\vect s}(D)$) and for every $(\zeta,z)\in D$, thanks to~\cite[Propositions 2.41 and 3.13]{CalziPeloso}. Therefore, the assertion follows by continuity.

(ii) By Lemma~\ref{lem:8} and~\cite[Proposition 2.41]{CalziPeloso}, we may assume that $B^{\vect{s'}}_{(\zeta,z)}\in A^{p,q}_{\vect s,0}(D)$ (resp.\ $B^{\vect{s'}}_{(\zeta,z)}\in A^{p,q}_{\vect s}(D)$) for every $(\zeta,z)\in D$. Then,
\[
(T_{\mi,\vect{s'}} B_{(\zeta,z)}^{\vect{s'}})(\zeta,z)=\norm{B^{2\vect{s'}}_{(\zeta,z)}}_{L^1(\mi)}
\]
for every $(\zeta,z)\in D$. In addition, observe that $A^{p,q}_{\vect b+\vect d-\vect{s'}-\vect{s''}}(D)\subseteq A^{\infty,\infty}_{(1-1/p)(\vect b+\vect d)-\vect{s'}-\vect{s''}}(D)$ by~\cite[Proposition 3.2]{CalziPeloso}, so that by means of~\cite[Proposition 2.41]{CalziPeloso} we see that there is a constant $C_1>0$ such that
\[
(T_{\mi,\vect{s'}} B_{(\zeta,z)}^{\vect{s'}})(\zeta,z)\meg C_1 \Delta_\Omega^{\vect s+2\vect{s'}+\vect{s''}-(\vect b+\vect d)}(\rho(\zeta,z))
\] 
for every $(\zeta,z)\in D$. Furthermore, by means of~\cite[Theorem 2.47]{CalziPeloso}, we see that there is a constant $C_2>0$ such that
\[
\norm{B^{2\vect{s'}}_{(\zeta,z)}}_{L^1(\mi)}\Meg C_2 B^{2 \vect{s'}}_{(\zeta,z)}(\zeta,z)M_1(\mi)(\zeta,z) = C_2  \Delta_\Omega^{2 \vect{s'}}(\rho(\zeta,z))M_1(\mi)(\zeta,z)
\]
for every $(\zeta,z)\in D$. It then follows that
\[
\Delta_\Omega^{\vect b+\vect d-\vect s-\vect{s''}}(\rho(\zeta,z))M_1(\mi)\meg \frac{C_1}{C_2} 
\]
for every $(\zeta,z)\in D$.

(iii) Arguing as in the proof of (ii), we may assume that $B^{\vect{s'}}_{(\zeta,z)}\in A^{p,q}_{\vect s,0}(D)$ for every $(\zeta,z)\in D$, and that $A^{p',q'}_{\vect b+\vect d-\vect s-\vect{s'}}(D)=\widetilde A^{p',q'}_{\vect b+\vect d-\vect s-\vect{s'}}(D)$ (cf.~\cite[Corollary 5.11]{CalziPeloso}).
Let us prove that   
\[
b_{(\zeta,z)}^{\vect{s'}}\coloneqq \Delta_\Omega^{(\vect b+\vect d)/p-\vect s-\vect{s'}}(\rho(\zeta,z)) B^{\vect{s'}}_{(\zeta,z)}\to 0,
\]
as $(\zeta,z)\to \infty$, in the weak topology $\sigma(A^{p,q}_{\vect s,0}(D), A^{p,q}_{\vect s,0}(D)')$.
Observe first that, by~\cite[Proposition 5.12]{CalziPeloso}, we may identify $A^{p,q}_{\vect s,0}(D)'$ with $A^{p',q'}_{\vect b+\vect d-\vect s-\vect{s'}}(D)$ by means of the sesquilinear form
\[
\langle\,\cdot\,\vert \,\cdot\,\rangle_{\vect{s'}}\colon A^{p,q}_{\vect s,0}(D)\times A^{p',q'}_{\vect b+\vect d-\vect s-\vect{s'}}(D)\ni (f,g)\mapsto \int_D f\overline g (\Delta_\Omega^{-\vect{s'}}\circ \rho)\,\dd \nu_D. 
\]
Now,
\[
\langle B^{\vect{s'}}_{(\zeta,z)}\Big\vert f\rangle_{\vect{s'}}= \frac{1}{c_{\vect{s'}}} \overline{(P_{\vect{s'}}f)(\zeta,z)}= \overline{f(\zeta,z)}
\]
for every $f\in A^{p',q'}_{\vect b+\vect d-\vect s-\vect{s'}}(D)$. In addition, $A^{p',q'}_{\vect b+\vect d-\vect s-\vect{s'}}(D)\subseteq A^{\infty,\infty}_{(1-1/p)(\vect b+\vect d)-\vect s-\vect{s'},0}(D)$ by~\cite[Proposition 3.7]{CalziPeloso}, since $p,q>1$. Therefore,
\[
\lim_{(\zeta,z)\to \infty}\langle b^{\vect{s'}}_{(\zeta,z)}\Big\vert f\rangle_{\vect{s'}}=0
\]
for every $f\in A^{p',q'}_{\vect b+\vect d-\vect s-\vect{s'}}(D)$, whence our assertion. 
Since $T_{\mi,\vect{s'}}$ is compact, this implies that 
\[
\lim_{(\zeta,z)\to \infty} \norm{T_{\mi,\vect{s'}}b^{\vect{s'}}_{(\zeta,z)}}_{A^{p,q}_{\vect b+\vect d-\vect {s'}-\vect{s''}}(D)}=0,
\]
so that the estimates of (ii) show that $M_1(\mi)\in L^{\infty,\infty}_{\vect b+\vect d-\vect{s'}-\vect{s''},0}(D)$.
\end{proof}

\begin{teo}\label{prop:14}
Take $p,q\in [1,\infty]$ and $\vect s,\vect{s'},\vect{s''}\in \R^r$ such that $\vect s\in \frac{1}{2 q} \vect m+(\R_+^*)^r$ if $q<\infty$ and $\vect s\in \R_+^r  $ if $q=\infty$, and such that $\vect{s'}\in \vect b+\vect d-\frac 1 2 \vect m-(\R_+^*)^r$.
Assume that $P_{\vect{s'}}$ induces a continuous linear mapping of $L^{p',q'}_{\vect{s''},0}(D)$ into $A^{p',q'}_{\vect{s''}}(D)$, and that $M_1(\mi)\in L^{\infty,\infty}_{\vect b+\vect d-\vect{s}-\vect{s''}}(D)$ (resp.\ $M_1(\mi)\in L^{\infty,\infty}_{\vect b+\vect d-\vect{s}-\vect{s''},0}(D)$).

Then, $T_{\mi,\vect{s'}}$ induces a continuous (resp.\ compact) linear mapping of $A^{p,q}_{\vect s}(D)$  into $A^{p,q}_{\vect b+\vect d-\vect{s'}-\vect{s''}}(D)$, and
\[
c_{\vect{s'}}\int_D (T_{\mi,\vect{s'}} f) \overline g (\Delta_\Omega^{-\vect{s'}}\circ \rho)\,\dd \nu_D= \int_D f \overline{P_{\vect{s'}}g}\,\dd \mi
\]
for every $f\in A^{p,q}_{\vect s}(D)$ and for every $g\in L^{p',q'}_{\vect{s''}}(D)$.
\end{teo}

This extends one implication of~\cite[Lemma 4.1]{NanaSehba}, where the case in which $\vect s=\vect{s''}\in \R \vect 1_r$, $\vect{s'}=\vect d-2 \vect s$, $p=q=2$, and $D$ is an irreducible symmetric tube domain, is considered.

\begin{proof}
Observe first that~\cite[Theorem 5.4]{CPCarleson} shows that  there is a constant $C_1>0$ such that
\[
\norm{f}_{L^1(\mi)}\meg C_1 \norm{f}_{A^{1,1}_{\vect s+\vect{s''}}(D)}
\]
for every $f\in A^{1,1}_{\vect s+\vect{s''}}(D)$. Let us first prove that, if $f\in A^{p,q}_{\vect s}(D)$, then $T_{\mi, \vect{s'}} f$ is a well defined element of $\Hol(D)$.
Indeed, since $B^{\vect{s'}}_{(\zeta,z)}\in A^{p',q'}_{\vect{s''}}(D)$ by~\cite[Propositions 2.41 and 5.20]{CalziPeloso}, it is clear that $f B^{\vect{s'}}_{(\zeta,z)}\in A^{1,1}_{\vect s+\vect{s''}}\subseteq L^1(\mi)$, so that $f B^{\vect{s'}}_{\,\cdot\,}(\zeta,z)\in L^1(\mi)$ for every $(\zeta,z)\in D$. The assertion the follows by Lemma~\ref{lem:2}. 

Next, take $C_2>0$ so that $\norm{P_{\vect{s'}}g}_{A^{p',q'}_{\vect{s''}}(D)}\meg C_2 \norm{g}_{L^{p',q'}_{\vect{s''}}(D)}$ for every $g\in L^{p',q'}_{\vect{s''},0}(D)$.
Then, take $g\in C_c(D)$ and observe that, by Fubini's theorem,
\[
\begin{split}
&\abs*{\int_D (T_{\mi,\vect{s'}} f)(\zeta,z) \overline {g(\zeta,z)} \Delta_\Omega^{-\vect{s'}}(\rho(\zeta,z))\,\dd \nu_D(\zeta,z)}\\
&\qquad=\abs*{\int_D \int_D B^{\vect{s'}}_{(\zeta',z')}(\zeta,z) f(\zeta',z')\,\dd \mi(\zeta',z') \overline {g(\zeta,z)} \Delta_\Omega^{-\vect{s'}}(\rho(\zeta,z))\,\dd \nu_D(\zeta,z)}\\
&\qquad=\abs*{\int_D f(\zeta',z')\overline{\int_D B^{\vect{s'}}_{(\zeta,z)}(\zeta',z') g(\zeta,z) \Delta_\Omega^{-\vect{s'}}(\rho(\zeta,z))\,\dd \nu_D(\zeta,z)} \,\dd \mi(\zeta',z')  }\\
&\qquad=\frac{1}{c_{\vect{s'}}}\abs*{\int_D f(\zeta',z')\overline{P_{\vect{s'}}g(\zeta',z')} \,\dd \mi(\zeta',z') }\\
&\qquad\meg \frac{1}{c_{\vect{s'}}} \norm{f P_{\vect{s'}}g}_{L^1(\mi)}\\
&\qquad\meg \frac{C_1}{c_{\vect{s'}}}\norm{f P_{\vect{s'}}g}_{A^{1,1}_{\vect s+\vect{s''} }(D)}\\
&\qquad\meg \frac{C_1}{c_{\vect{s'}}}  \norm{f}_{A^{p,q}_{\vect s}(D)} \norm{P_{\vect{s'}}g}_{A^{p',q'}_{\vect{s''}}(D)},\\
&\qquad\meg \frac{C_1 C_2}{c_{\vect{s'}}}\norm{f}_{A^{p,q}_{\vect s}(D)} \norm{g}_{L^{p',q'}_{\vect{s''}}(D)}.
\end{split}
\]
Therefore, $T_{\mi,\vect{s'}} $ induces a continuous linear mapping of $A^{p,q}_{\vect s}(D)$ into $A^{p,q}_{\vect b+\vect d-\vect{s'}-\vect{s''}}(D)$.

Now, assume that $M_1(\mi)\in L^{\infty,\infty}_{\vect b+\vect d-\vect{s'}-\vect{s''},0}(D)$. Observe that, in order to prove that  $T_{\mi,\vect{s'}} $ induces a compact linear mapping of $A^{p,q}_{\vect s}(D)$ into $A^{p,q}_{\vect b+\vect d-\vect{s'}-\vect{s''}}(D)$, by means of the preceding computations we may reduce to the case in which $\mi$ is $\neq0$ and has compact support in $D$. Let $R$ be the diameter of $\Supp{\mi}$.
Let $(f_j)_{j\in \N}$ be a  bounded sequence in $A^{p,q}_{\vect s}(D)$, and let us prove that $(T_{\mi,\vect{s'}}f_j)$ has a convergent subsequence. Observe first that, since $A^{p,q}_{\vect s}(D)$ embeds continuously into the Fréchet--Montel space $\Hol(D)$, we may assume that $(f_j)$ converges to some $f\in A^{p,q}_{\vect s}(D)$ locally uniformly. Up to replacing $(f_j)$ with $(f_j-f)$, we may therefore assume that $f=0$. 
Now, observe that, by~\cite[Theorem 2.47]{CalziPeloso}, there is a constant $C_3>0$ such that
\[
\frac{1}{C_3} \abs*{B_{(\zeta,z)}^{\vect{s'}}(\zeta',z')}\meg \abs*{B_{(\zeta,z)}^{\vect{s'}}(\zeta'',z'')}\meg C_3\abs*{B_{(\zeta,z)}^{\vect{s'}}(\zeta',z')}
\]
for every $(\zeta,z),(\zeta',z'),(\zeta'',z'')\in D$ such that $d((\zeta',z'),(\zeta'',z''))\meg R$. Hence,
\[
\abs{T_{\mi,\vect{s'}} f_j}\meg C_3 \abs{B_{(\zeta,z)}^{\vect{s'}}} \abs{\mi}(D) \norm{\chi_{\Supp{\mi}}f_j}_{L^\infty(D)},
\]
where $(\zeta,z)$ is a (fixed) element of $\Supp{\mi}$. It then follows that
\[
\norm{T_{\mi,\vect{s'}} f_j}_{A^{p,q}_{\vect b+\vect d-\vect {s'}-\vect{s''}}(D)}\meg C_3 \norm{B_{(\zeta,z)}^{\vect{s'}}}_{A^{p,q}_{\vect b+\vect d-\vect {s'}-\vect{s''}}(D)} \abs{\mi}(D) \norm{\chi_{\Supp{\mi}}f_j}_{L^\infty(D)},
\]
so that $(T_{\mi,\vect{s'}} f_j)$ converges to $0$ in $A^{p,q}_{\vect b+\vect d-\vect {s'}-\vect{s''}}(D)$.\footnote{Notice that $B_{(\zeta,z)}^{\vect{s'}}\in A^{p,q}_{\vect b+\vect d-\vect {s'}-\vect{s''}}(D)$ since $T_{\delta_{(\zeta,z)},\vect{s'}}\colon A^{p,q}_{\vect s}(D)\to A^{p,q}_{\vect b+\vect d-\vect{s'}-\vect{s''}}(D)$ is continuous by the preceding computations.} The assertion follows by the arbitrariness of $(f_j)$.
\end{proof}

 We now recall the definition of Schatten classes.  For lack of a
precise reference, in 
Section~\ref{sec:Schatten} we collect and describe  the main facts that
we use.

\begin{deff}\label{def:Schatten-classes}
Let $H_1,H_2$ be two hilbertian spaces, and take $p\in
]0,\infty[$. Then, we define $\Lin^p(H_1;H_2)=\Lin^p_0(H_1;H_2)$ as
the space of  $T\in \Lin(H_1;H_2)$ such that  
\[
\norm{T}_{\Lin^p(H_1;H_2)}\coloneqq (\tr((T^*T)^{p/2}))^{1/p}=(\tr((T T^*)^{p/2}))^{1/p}<\infty.
\]
We also define $\Lin^\infty(H_1;H_2)\coloneqq \Lin(H_1;H_2)$, and
$\Lin^\infty_0(H_1;H_2)$ as the space of compact linear operators from
$H_1$ into $H_2$. 
\end{deff}

\begin{teo}\label{prop:33}
Take $p\in ]0,\infty[$ and $\vect s,\vect{s'}, \vect{s''}\in \R^r$ such that $\vect s,\vect{s''},\vect b+\vect d-\vect{s'}-\vect{s''}\in \frac 1 4 \vect m+(\R_+^*)^r$.
Consider the following conditions:
\begin{enumerate}
\item[{\em(1)}] $T_{\mi,\vect{s'}}$ induces an element of $\Lin^p(A^{2,2}_{\vect s}(D);A^{2,2}_{\vect b+\vect d-\vect {s'}-\vect{s''}}(D))$;

\item[{\em(2)}]  $M_1(\mi)\in L^{p,p}_{(1+1/p)(\vect b+\vect d)- \vect {s}-\vect{s''}}(D)$.
\end{enumerate}
Then, {\em(2)} implies {\em(1)}. 
If, in addition, $\mi$ is positive and $\vect s = \vect{s''}$ when $p<1$, then {\em(1)} implies {\em(2)}. 
\end{teo}

This extends~\cite[Theorem 4.2]{NanaSehba}, where the case in which $\vect s=\vect{s''}\in \R \vect 1_r$, $\vect{s'}=\vect d-2 \vect s$, and $D$ is an irreducible symmetric tube domain, is considered.
This also extends~\cite[Theorem 2.1]{LiLuecking}, where the case in which $D$ is a strongly pseudoconvex domain is considered.\footnote{Note that a homogeneous Siegel domain is strongly pseudoconvex if and only if $r=1$, since the Shilov boundary of a strongly pseudoconvex domain is its topological boundary (cf., e.g.,~\cite[Theorem 15.3]{Fuks}).}

Before we pass to the proof, we need two lemmas. We define $i^{\vect s}= e^{i \frac{\pi}{2} (s_1+\cdots+s_r)}$ for every $\vect s\in \C^r$ to simplify the notation.

\begin{lem}\label{lem:7}
Take $\vect s,\vect{s'}\in \R^r$ such that $\vect s,\vect s-\vect{s'}\in \frac{1}{4}\vect{m}+(\R_+^*)^r$. Then, the following hold:
\begin{enumerate}
\item[{\em (1)}] denoting by $\Ic$ the isomorphism of $ A^{2,2}_{\vect s}(D)$ onto $ A^{2,2}_{\vect s-\vect{s'}}(D)$ which induces the endomorphism $f\mapsto f*I^{\vect{s'}}_\Omega$ of $\Ec(\Sc_{\Omega,L}(\Nc))$  (cf.~\cite[Proposition 5.13]{CalziPeloso}), 
\[
\Ic\big(B^{\vect b+\vect d-2 \vect s}_{(\zeta,z)}\big)= \frac{2^{\vect{s'}}\Gamma_{\Omega'}(2\vect s-\vect{s'}-\vect b-\vect d)}{i^{\vect{s'}}\Gamma_{\Omega'}(2\vect s-\vect b-\vect d)} B^{\vect b+\vect d-2 \vect s+\vect{s'}}_{(\zeta,z)} 
\]
for every $(\zeta,z)\in D$;

\item[{\em(2)}] for every $f\in \Ec(\Sc_{\Omega,L}(\Nc))$,
\[
\langle f_h*I^{\vect{s'}}_\Omega\Big\vert \big(B^{\vect b+\vect d-2\vect s}_{(\zeta,z)}\big)_h\rangle_{L^2(\Nc)}= \frac{2^{\vect{s'}}\Gamma_{\Omega'}(2\vect s-\vect{s'}-\vect b-\vect d)}{i^{\vect{s'}}\Gamma_{\Omega'}(2\vect s-\vect b-\vect d)} \langle f_h\Big\vert \big(B^{\vect b+\vect d-2\vect s+\vect{s'}}_{(\zeta,z)}\big)_h\rangle_{L^2(\Nc)}
\]
for every $(\zeta,z)\in D$, and for every $h\in\Omega$.
\end{enumerate}
\end{lem}

\begin{proof}
Observe that~\cite[Corollary 1.41 and Propositions 2.14 and 3.11]{CalziPeloso} show that
\[
\pi_\lambda\big( \big(B^{\vect b+\vect d-2\vect s}_{(\zeta,z)}\big)_h  \big)=  c'_{\vect s} \chi_{\Omega'}(\lambda) \Delta_{\Omega'}^{2\vect s}(\lambda)e^{-\langle\lambda, \rho(\zeta,z)+h \rangle} \pi_\lambda(\zeta,\Rea z)P_{\lambda,0}
\]
for almost every $\lambda\in F'\setminus W$, where
\[
c'_{\vect s}\coloneqq \frac{4^{\vect{s}} c}{ \Gamma_{\Omega'}(2\vect s-\vect b-\vect d) }
\]
for a suitable constant $c>0$.
Then, (1) follows by means of~\cite[Lemma 2.21]{CalziPeloso}, while (2) follows by means of~\cite[Proposition 4.11]{CalziPeloso}.
\end{proof}

\begin{lem}\label{lem:16}
Take $\vect s,\vect{s'}, \vect{s''},\vect{s'''}\in \R^r$ such that 
\[
\vect s,\vect{s''},\vect b+\vect d-\vect{s'}-\vect{s''}, \vect b+\vect d-\vect{s'}-\vect{s''}-\vect{s'''}\in \frac 1 4 \vect m+(\R_+^*)^r.
\] 
Take $\mi\in \cM(D)$ such that $M_1(\mi)\in L^{\infty,\infty}_{\vect b+\vect d-\vect s-\vect{s''},0}(D)$, and denote by $\Ic$ the isomorphism of $A^{2,2}_{\vect b+\vect d-\vect{s'}-\vect{s''}}(D)$ onto $ A^{2,2}_{\vect b+\vect d-\vect{s'}-\vect{s''}-\vect{s'''}}(D)$ which induces the automorphism $f \mapsto f*I^{\vect{s'''}}_\Omega$ of $\Ec(\Sc_{\Omega,L}(\Nc))$ (cf.~\cite[Proposition 5.13]{CalziPeloso}). Then,
\[
\Ic T_{\mi,\vect{s'}}=T_{\mi,\vect{s'}+\vect{s'''}}.
\]
\end{lem}

Notice that the assertion is contained in Lemma~\ref{lem:8} if $\vs'''\in-\N_{\Omega'}$. The proof below is more delicate since $\Ic$ is no longer a differential operator.

\begin{proof}
By Lemma~\ref{lem:7}, the assertion is clear if $\mi$ has finite support. Now, assume that $\mi$ has compact support and observe that there is a bounded filter $\Ff$ on the space of measures on the with finite support which converges vaguely to $\mi$ (cf.~\cite[Corollary 1 to Theorem 1 of Chapter III, § 2, No.\ 4]{BourbakiInt1}). We may further assume that there is $\Ms\in \Ff$ such that every element of $\Ms$ is supported in $\Supp{\mi}$. Therefore, it is clear that $T_{\mi',\vect{s'}}f$ and $T_{\mi',\vect{s'},\vect{s'''}}f$ converge locally uniformly to $T_{\mi,\vect{s'}}f$ and $T_{\mi,\vect{s'},\vect{s'''}}f$, respectively, as $\mi'$ runs along $\Ff$, for every $f\in A^{2,2}_{\vect s}(D)$. Since they stay also bounded along $\Ff$, it is clear that $T_{\mi',\vect{s'}}f$ and $T_{\mi',\vect{s'},\vect{s'''}}f$ converge weakly to $T_{\mi,\vect{s'}}f$ and $T_{\mi,\vect{s'},\vect{s'''}}f$, respectively, as $\mi'$ runs along $\Ff$. Hence, $\Ic T_{\mi',\vect{s'}} f$ converges weakly to $\Ic T_{\mi,\vect{s'}}f$ as $\mi'$ runs along $\Ff$, whence
\[
\Ic T_{\mi,\vect{s'}}f= T_{\mi, \vect{s'}+\vect{s'''}}f
\]
for every $f\in A^{2,2}_{\vect s}(D)$.

Then, take $\mi\in \cM(D)$ such that $M_1(\mi)\in L^{\infty,\infty}_{\vect b+\vect d- \vect {s}-\vect{s''},0}(D)$. Observe that, if we define $\mi_\ell\coloneqq \chi_{B((0,i e_\Omega),\ell)}\cdot\mi$ for every $\ell\in \N^*$, then $\norm{M_1(\mi-\mi_\ell)}_{L^{\infty,\infty}_{\vect b+\vect d-\vect s-\vect{s''}}(D)}\to 0$ as $\ell\to \infty$, so that $T_{\mi_\ell,\vect{s'}}  $ and $T_{\mi_\ell,\vect{s'}+\vect{s'''}}$ converge to $T_{\mi,\vect{s'}}  $ and $T_{\mi,\vect{s'}+\vect{s'''}}$ in $\Lin(A^{2,2}_{\vect s}(D);A^{2,2}_{\vect b+\vect d-\vect{s'}-\vect{s''}})$ and $\Lin(A^{2,2}_{\vect s}(D);A^{2,2}_{\vect{s''}})$, respectively, by Theorem~\ref{prop:14}. The assertion follows.
\end{proof}

\begin{proof}[Proof of Theorem~\ref{prop:33}.]
(2) $\implies$ (1). Assume first that $p\in ]0,1]$. Define $\vect{s_1}\coloneqq \vect b+\vect d-\vect{s'}-2 \vect{s''}$, so that the automorphism $f\mapsto f*I^{\vect{s_1}}_\Omega$ of $\Ec(\Sc_{\Omega,L}(\Nc))$ induces an isomorphism $\Ic$ of $A^{2,2}_{\vect b+\vect d-\vect{s'}-\vect{s''}}(D)$ onto $A^{2,2}_{\vect{s''}}(D)$ (cf.~\cite[Proposition 5.13]{CalziPeloso}). Observe that, by Lemma~\ref{lem:16},
\[
\Ic T_{\mi,\vect{s'}}= T_{\mi, \vect{s'}+\vect{s_1}}.
\]

Then, take $\vect{s_2}\in \R^r$ and a $(\delta,R)$-lattice $(\zeta_{j,k},z_{j,k})_{j\in J,k\in K}$ on $D$, for some $\delta>0$ and some $R>1$, such that the mappings
\[
A_1\colon \ell^{2,2}(J,K)\ni \lambda \mapsto \sum_{j,k} \lambda_{j,k} B^{\vect{s_2}}_{(\zeta_{j,k},z_{j,k})} \Delta_\Omega^{(\vect b+\vect d)/2-\vect s-\vect{s_2}}(h_k)\in A^{2,2}_{\vect s}(D)
\]
and
\[
A_2\colon \ell^{2,2}(J,K)\ni \lambda \mapsto \sum_{j,k}\lambda_{j,k} B^{\vect{s_2}}_{(\zeta_{j,k},z_{j,k})} \Delta_\Omega^{(\vect b+\vect d)/2-\vect{s''}-\vect{s_2}}(h_k)\in A^{2,2}_{\vect{s''}}(D)
\]
are continuous and have a continuous linear section, where
$h_k\coloneqq \rho(\zeta_{j,k},z_{j,k})$ for every $(j,k)\in J\times
K$ (cf.~\cite[Proposition 3.15 and Corollary 5.16]{CalziPeloso}).  
Then, Proposition~\ref{prop:15} implies that there is a constant $C_1>0$ such that
\[
\norm{T_{\mi,\vect{s'}}}_{\Lin^p(A^{2,2}_{\vect s}(D);A^{2,2}_{\vect b+\vect d-\vect{s'}-\vect{s''}}(D))}\meg C_1 \norm{A_2^* T_{\mi,\vect{s'}+\vect{s_1}}A_1}_{\Lin^p(\ell^{2,2}(J,K))}.
\] 
In addition, Proposition~\ref{prop:34} shows that
\[
\norm{A_2^* T_{\mi,\vect{s'}+\vect{s_1}}A_1}_{\Lin^p(\ell^{2,2}(J,K))}^p\meg \sum_{(j,k),(j',k')\in J\times K} \abs*{\langle T_{\mi,\vect{s'}+\vect{s_1}} A_1 e_{j,k}\vert A_2 e_{j',k'}\rangle_{A^{2,2}_{\vect{s''}}(D)}  }^p,
\]
where $(e_{j,k})$ is the orthonormal basis of $\ell^{2,2}(J,K)$ defined by $e_{j,k}(j',k')\coloneqq \delta_{(j,k),(j',k')}$ for every $(j,k),(j',k')\in J\times K$.
Now,
\[
\begin{split}
&\abs*{\langle T_{\mi,\vect{s'}+\vect{s_1}} A_1 e_{j,k}\vert A_2 e_{j',k'}\rangle_{A^{2,2}_{\vect{s''}}(D)}  }^p\\
&\qquad= \Delta_\Omega^{p[(\vect b+\vect d)/2-\vect s-\vect{s_2}]}(h_k)\Delta_\Omega^{p[(\vect b+\vect d)/2-\vect {s''}-\vect{s_2}]}(h_{k'})\abs*{\langle T_{\mi,\vect{s'}+\vect{s_1}} B^{\vect{s_2}}_{(\zeta_{j,k},z_{j,k})} \Big\vert B^{\vect{s_2}}_{(\zeta_{j',k'},z_{j',k'})}\rangle_{A^{2,2}_{\vect{s''}}(D)} }^p\\
&\qquad=\frac{1}{c_{\vect{s'}+\vect{s_1}}^p}\Delta_\Omega^{p[(\vect b+\vect d)/2-\vect s-\vect{s_2}]}(h_k)\Delta_\Omega^{p[(\vect b+\vect d)/2-\vect {s''}-\vect{s_2}]}(h_{k'})\abs*{\int_{D} B^{\vect{s_2}}_{(\zeta_{j,k},z_{j,k})} \overline{B^{\vect{s_2}}_{(\zeta_{j',k'},z_{j',k'})}}\,\dd \mi }^p
\end{split}
\]
for every $(j,k),(j',k')\in J\times K$,
where the second equality follows from Theorem~\ref{prop:14}, since $2\vect{s''}=\vect b+\vect d-\vect{s'}-\vect{s_1}$ by our auxiliary assumption. Now, let $(B_{j,k})$ be a Borel partition of $D$ such that $B((\zeta_{j,k},z_{j,k}),\delta)\subseteq B_{j,k}\subseteq B((\zeta_{j,k},z_{j,k}),R\delta)$ for every $(j,k)\in J\times K$. In addition, by~\cite[Theorem 2.42 and Corollary 2.44]{CalziPeloso}, we may take a constant $C_2>0$ such that
\[
\frac{1}{C_2} \abs{B^{\vect{s_2}}_{(\zeta,z)}(\zeta',z')}\meg  \abs{B^{\vect{s_2}}_{(\zeta,z)}(\zeta'',z'')}\meg {C_2} \abs{B^{\vect{s_2}}_{(\zeta,z)}(\zeta',z')}
\]
for every $(\zeta,z),(\zeta',z'),(\zeta'',z'')\in D$ such that $d((\zeta',z'),(\zeta'',z''))<R\delta$,
and such that
\[
\frac{1}{C_2}\Delta_\Omega^{(\vect b+\vect d)/2-\vect s-\vect{s_2}}(h)\meg \Delta_\Omega^{(\vect b+\vect d)/2-\vect s-\vect{s_2}}(h')\meg C_2\Delta_\Omega^{(\vect b+\vect d)/2-\vect s-\vect{s_2}}(h)
\]
and
\[
\frac{1}{C_2}\Delta_\Omega^{(\vect b+\vect d)/2-\vect{s''}-\vect{s_2}}(h)\meg \Delta_\Omega^{(\vect b+\vect d)/2-\vect {s''}-\vect{s_2}}(h')\meg C_2\Delta_\Omega^{(\vect b+\vect d)/2-\vect{s''}-\vect{s_2}}(h)
\]
for every $h,h'\in \Omega$ such that $d(h,h')<R\delta$.
Then,
\[
\begin{split}
\abs*{\int_{D} B^{\vect{s_2}}_{(\zeta_{j,k},z_{j,k})} \overline{B^{\vect{s_2}}_{(\zeta_{j',k'},z_{j',k'})}}\,\dd \mi }^p&\meg C_2^{2p}\left( \sum_{j'',k''}\abs{\mi}(B_{j'',k''}) \abs{B^{\vect{s_2}}_{(\zeta_{j,k},z_{j,k})}(\zeta_{j'',k''}) B^{\vect{s_2}}_{(\zeta_{j',k'},z_{j',k'})}(\zeta_{j'',k''})} \right)^p\\
&\meg C_2^{2p}\sum_{j'',k''}\abs{\mi}(B_{j'',k''})^p \abs{B^{p\vect{s_2}}_{(\zeta_{j,k},z_{j,k})}(\zeta_{j'',k''}) B^{p\vect{s_2}}_{(\zeta_{j',k'},z_{j',k'})}(\zeta_{j'',k''})} 
\end{split}
\]
for every $(j,k),(j',k')\in J\times K$, since $p\meg 1$. Hence,
\[
\begin{split}
\norm{A_2^* T_{\mi,\vect{s'}+\vect{s_1}}A_1}_{\Lin^p(\ell^{2,2}(J,K))}^p&\meg \frac{C_2^{2p}}{c_{\vect{s'}+\vect{s_1}}^p}\sum_{j'',k''}\abs{\mi}(B_{j'',k''})^p \sum_{j,k} \Delta_\Omega^{p[(\vect b+\vect d)/2-\vect s-\vect{s_2}]}(h_k)\abs{B^{p\vect{s_2}}_{(\zeta_{j,k},z_{j,k})}(\zeta_{j'',k''})}\\
&\qquad\times\sum_{j',k'} \Delta_\Omega^{p[(\vect b+\vect d)/2-\vect {s''}-\vect{s_2}]}(h_{k'})\abs{B^{p\vect{s_2}}_{(\zeta_{j',k'},z_{j',k'})}(\zeta_{j'',k''})}\\
&\meg \frac{C_2^{6p}}{c_{\vect{s'}+\vect{s_1}}^p\nu_D(B((0,ie_\Omega),\delta))^2}\sum_{j'',k''}\abs{\mi}(B_{j'',k''})^p \norm{B^{p\vect{s_2}}_{(\zeta_{j'',k''})}}_{A^{1,1}_{p[(\vect b+\vect d)/2-\vect s-\vect{s_2}]+\vect b+\vect d}(D)}\\
&\qquad \times \norm{B^{p\vect{s_2}}_{(\zeta_{j'',k''})}}_{A^{1,1}_{p[(\vect b+\vect d)/2-\vect {s''}-\vect{s_2}]+\vect b+\vect d}(D)} .
\end{split}
\]
Now, by means of~\cite[Proposition 2.41]{CalziPeloso} we see that, provided that $\vect{s_2}$ is sufficiently small, there is a constant $C_3>0$ such that
\[
\norm{T_{\mi,\vect{s'}}}_{\Lin^p(A^{2,2}_{\vect s}(D);A^{2,2}_{\vect b+\vect d-\vect{s'}-\vect{s''}}(D))}^p\meg C_3\sum_{j'',k''}\Delta_\Omega^{p[\vect b+\vect d-\vect s-\vect{s''}]}(h_{k''})\abs{\mi}(B_{j'',k''})^p= C_3 \norm{\mi}_p^p,
\]
where 
\[
\norm{\mi}_p\coloneqq \left( \sum_{j,k}\Delta_\Omega^{p[\vect b+\vect d-\vect s-\vect{s''}]}(h_{k})\abs{\mi}(B_{j,k})^p \right)^{1/p},
\]
with the obvious modification when $p=\infty$. 

In addition, there is a constant $C_4>0$ such that
\[
\norm{T_{\mi,\vect{s'}}}_{\Lin(A^{2,2}_{\vect s}(D);A^{2,2}_{\vect b+\vect d-\vect{s'} -\vect{s''}}(D))}\meg C_4 \norm{\mi}_{\infty} 
\]
for every Radon measure $\mi$ on $D$ such that $\norm{\mi}_{\infty} <\infty$, as one sees by inspection of the proof of~\cite[Theorem 5.4]{CPCarleson} and of Theorem~\ref{prop:14}.

Now, take $p\in ]1,\infty[$ and assume that $\mi$ has compact support. Define
\[
\mi_z\coloneqq \norm{\mi}_{p}^{p(z-1)+1}\sum_{j,k} [\Delta_\Omega^{\vect{s'}}(h_k)\abs{\mi}(B_{j,k})]^{p (1-z)-1}\chi_{B_{j,k}}\cdot \mi, 
\]
and observe that the mapping
\[
F\colon z \mapsto T_{\mi_z,\vect{s'}}=\norm{\mi}_{p}^{p(z-1)+1}\sum_{j,k}[\Delta_\Omega^{\vect{s'}}\abs{\mi}(B_{j,k})]^{p (1-z)-1} T_{\chi_{B_{j,k}}\cdot \mi,\vect b+\vect d-2 \vect s}  \in \Lin(A^{2,2}_{\vect s}(D);A^{2,2}_{\vect b+\vect d-\vect {s'}-\vect{s''}}(D))
\]
is holomorphic on $\C$ and bounded on the closure of $S\coloneqq \Set{z\in \C\colon 0<\Rea z<1}$. In addition, $F(1/p')=T_{\mi,\vect{s'}}$,
\[
\norm{F(i t)}_{\Lin^1(A^{2,2}_{\vect s}(D);A^{2,2}_{\vect b+\vect d-\vect {s'}-\vect{s''}}(D))}\meg C_3 \norm{\mi_{it}}_{1}=C_3\norm{\mi}_{p}
\]
and 
\[
\norm{F(1+i t)}_{\Lin(A^{2,2}_{\vect s}(D);A^{2,2}_{\vect b+\vect d-\vect {s'}-\vect{s''}}(D))}\meg C_4 \norm{\mi_{1+it}}_{\infty}=C_4 \norm{\mi}_{p}.
\]
Therefore, Proposition~\ref{prop:19} implies that
\[
\norm{T_{\mi,\vect{s'}}}_{\Lin^p(A^{2,2}_{\vect s}(D);A^{2,2}_{\vect b+\vect d-\vect {s'}-\vect{s''}}(D))}\meg \max(C_3,C_4)\norm{\mi}_{p}.
\]

Finally, take $p\in ]1,\infty[$, and let $\mi$ be a Radon measure on $D$ such that $\norm{\mi}_{p}<\infty$. Define
\[
\mi_\ell\coloneqq \chi_{B((0,i e_\Omega), \ell)}\cdot \mi
\]
for every $\ell\in \N^*$, and observe that the preceding remarks show that
\[
\norm{T_{\mi_\ell,\vect{s'}}}_{\Lin^p(A^{2,2}_{\vect s}(D))} \meg\max(C_3,C_4)\norm{\mi}_{p}
\]
for every $\ell\in \N^*$. In addition, it is clear that $\norm{\mi-\mi_\ell}_{\infty}\to 0$ for $\ell\to \infty$, so that $T_{\mi_\ell,\vect{s'}}\to T_{\mi,\vect{s'}}$ in $\Lin(A^{2,2}_{\vect s}(D);A^{2,2}_{\vect b+\vect d-\vect{s'}-\vect{s''}}(D))$ for $\ell \to \infty$. Hence, Proposition~\ref{prop:17} shows that
\[
\norm{T_{\mi,\vect{s'}}}_{\Lin^p(A^{2,2}_{\vect s}(D);A^{2,2}_{\vect b+\vect d-\vect {s'}-\vect{s''}}(D))}\meg\max(C_3,C_4)\norm{\mi}_{p},
\]
whence the conclusion thanks to~\cite[Lemma 5.1]{CPCarleson}.

(1) $\implies$ (2). Assume first that $\mi$ is positive, $p\meg 1$, and $\vect s=\vect{s''}$. Define $\vect{s_1}$ and $\Ic$ as in the proof of the implication (2) $\implies$ (1). We also define $A_1$ and $A_2$ similarly, except for the fact that this time $(\zeta_{j,k},z_{j,k})_{j\in J,k\in K}$ will be an $(R,4)$-lattice (cf.~\cite[Lemma 2.55]{CalziPeloso}), so that $A_1$ and $A_2$ are still continuous, but not necessarily onto (cf.~\cite[Propositions 3.17 and 3.32, Theorem 3.34, and the proof of Corollary 5.14]{CalziPeloso}), provided that $\vect{s_2}$ is sufficiently small.
We shall define $B_{j,k}\coloneqq B((\zeta_{j,k},z_{j,k}),R\delta )$ for every $(j,k)\in J\times K$. The precise conditions to be imposed on $\delta$ and $R$ will be determined later on.

Observe that $\Ic T_{\mi,\vect{s'}}=T_{\mi,\vect{s'}+\vect{s_1}}$ by Proposition~\ref{prop:11} and Lemma~\ref{lem:16}. 
\[
\langle \Ic T_{\mi,\vect{s'}}f \vert f \rangle_{A^{2,2}_{\vect{s}}(D)}= \frac{1}{c_{\vect{s'}+\vect{s_1}}} \int_D \abs{f}^2\,\dd \mi
\]
for every $f\in A^{2,2}_{\vect s}(D)$, by Proposition~\ref{prop:11}, so that, in particular, $\Ic T_{\mi,\vect{s'}}$ is (self-adjoint and) positive.  In addition, if $\mi'$ is a positive Radon measure on $D$ and $\mi'\meg \mi$, then $\langle \Ic T_{\mi',\vect{s'}}f \vert f \rangle_{A^{2,2}_{\vect{s}}(D)}\meg \langle \Ic T_{\mi,\vect{s'}}f \vert f \rangle_{A^{2,2}_{\vect{s}}(D)}$ for every $f\in A^{2,2}_{\vect s}(D)$, so that $(\Ic T_{\mi',\vect{s'}})^{1/2}$ is the composite of $(\Ic T_{\mi,\vect{s'}})^{1/2}$ with a contraction of $A^{2,2}_{\vect s}(D)$. Hence, 
\[
\norm{\Ic T_{\mi',\vect{s'}}}_{\Lin^p(A^{2,2}_{\vect s}(D))}\meg \norm{\Ic T_{\mi,\vect{s'}}}_{\Lin^p(A^{2,2}_{\vect s}(D))}
\]
for every $\mi'$ as above, thanks to Proposition~\ref{prop:15}.

Define $X_{\mi'}\coloneqq A_2^* \Ic T_{\mi',\vect{s'}}A_1$, and observe that $X_{\mi'}\in \Lin^p(\ell^{2,2}(J,K))$. Observe that $X_{\mi'}= A_2^*  T_{\mi',\vect{s'}+\vect{s_1}}A_1$ by Proposition~\ref{prop:11} and Lemma~\ref{lem:16}. 
In addition, define $\Delta_{\mi'}$ as the \emph{diagonal} operator whose diagonal elements are the same as those of $X_{\mi'}$. Then, the computations of the proof of the implication (2) $\implies$ (1) show that
\[
\begin{split}
\norm{\Delta_{\mi'}}_{\Lin^p(\ell^{2,2}(J,K))}^p &=\frac{1}{c_{\vect{s'}+\vect{s_1}}^p}\sum_{j,k} \Delta_\Omega^{p[\vect b+\vect d-2\vect s-2\vect{s_2}]}(h_k)\left( \int_D \abs{B_{(\zeta_{j,k},z_{j,k})}^{\vect{s_2}}}^2\,\dd \mi'\right)^p\\
&\Meg \frac{1}{c_{\vect{s'}+\vect{s_1}}^p C_2^{2p}}\sum_{j,k} \Delta_\Omega^{p[\vect b+\vect d-2\vect s]}(h_k) \mi'(B_{j,k})^p
\end{split}
\]
for every positive Radon measure $\mi'\meg \mi$ on $D$, and that
\[
\begin{split}
&\norm{X_{\mi'}-\Delta_{\mi'}}_{\Lin^p(\ell^{2,2}(J,K))}^p\meg \frac{C_2^{6p}}{c_{\vect{s'}+\vect{s_1}}^p\nu_D(B((0,i e_\Omega),R/2))} \sum_{j,k} \mi'(B_{j,k})^p \int_{d((\zeta,z),(\zeta',z'))>R} \abs{B^{p\vect{s_2}}_{(\zeta_{j'',k''},z_{j'',k''})}(\zeta,z)} \\
&\qquad \times\abs{B^{p\vect{s_2}}_{(\zeta_{j'',k''},z_{j'',k''})}(\zeta',z')}\Delta_\Omega^{p[(\vect b+\vect d)/2-\vect s-\vect{s_2}]}(\rho(\zeta,z))\Delta_\Omega^{p[(\vect b+\vect d)/2-\vect {s''}-\vect{s_2}]}(\rho(\zeta',z'))\,\dd (\nu_D\otimes \nu_D)((\zeta,z),(\zeta',z'))
\end{split}
\]
for every positive Radon measure $\mi'\meg \mi$ on $D$ such that $\mi'(D\setminus \bigcup_{j,k} B_{j,k})=0$,
where $C_2$ is defined as in the proof of the implication (2) $\implies$ (1) and is independent of $\delta$ and $R$ as long as (say) $R\delta<1$.
Now, observe  that there is a constant $C_R>0$ such that
\[
\begin{split}
&\int_{d((\zeta,z),(\zeta',z'))>R} \abs{B^{p\vect{s_2}}_{(\zeta_{j'',k''},z_{j'',k''})}(\zeta,z)B^{p\vect{s_2}}_{(\zeta_{j'',k''},z_{j'',k''})}(\zeta',z')}\Delta_\Omega^{p[(\vect b+\vect d)/2-\vect s-\vect{s_2}]}(\rho(\zeta,z))\\
&\qquad\times\Delta_\Omega^{p[(\vect b+\vect d)/2-\vect {s''}-\vect{s_2}]}(\rho(\zeta',z'))\,\dd (\nu_D\otimes \nu_D)((\zeta,z),(\zeta',z'))=C_R \Delta_\Omega^{p[\vect b+\vect d-2\vect{s}]}(h_{k''})
\end{split}
\]
for every $k''\in K$, by homogeneity, and that $\lim_{R\to \infty} C_R=0$, provided that $\vect{s_2}$ is sufficiently small (cf.~\cite[Proposition 2.41]{CalziPeloso}).
Therefore, we may find $\delta, R>0$ such that $\delta R<1$ and
\[
\norm{X_{\mi'}-\Delta_{\mi'}}_{\Lin^p(\ell^{2,2}(J,K))}\meg \frac{1}{4 c_{\vect{s'}+\vect{s_1}}^p C_2^{2p}} \sum_{j,k} \Delta_\Omega^{p[\vect b+\vect d-2\vect{s}]}(h_k) \mi'(B_{j,k})^p
\]
for every positive Radon measure $\mi'\meg \mi$ on $D$ such that $\mi'(D\setminus \bigcup_{j,k} B_{j,k})=0$. Therefore,~\cite[Lemma 9 (b) of Chapter IX, § 9]{DunfordSchwartz} implies that
\[
\norm{X_{\mi'}}_{\Lin^p(\ell^{2,2}(J,K))}\Meg  \frac{1}{4 c_{\vect{s'}+\vect{s_1}}^p C_2^{2p}} \sum_{j,k} \Delta_\Omega^{p[\vect b+\vect d-2\vect{s}]}(h_k) \mi'(B_{j,k})^p
\]
for every positive Radon measure \emph{with compact support} $\mi'\meg
\mi$ on $D$ such that $\mi'(D\setminus \bigcup_{j,k} B_{j,k})=0$.  By
approximation, the same holds for every positive Radon measure
$\mi'\meg \mi$ on $D$ such that $\mi'(D\setminus \bigcup_{j,k}
B_{j,k})=0$. 

Now, observe that, by the proof of~\cite[Lemma 2.55]{CalziPeloso} we may find two countable families of affine automorphisms $(\phi_{j'})_{j'\in J'}$ and $(\psi_{k'})_{k'\in K'}$ such that the following hold:
\begin{itemize}
\item[(a)] $\phi_{j'}$ is induced by the action of some element of $b D$ on $D$, for every $j'\in J'$;

\item[(b)] $\psi_{k'}= g_{k'}\times t_{k'}$ for some $t_{k'}\in T_+$ and $g_{k'}\in GL(E)$ such that $t_{k'}\circ \Phi=\Phi\circ (g_{k'}\times g_{k'})$, for every $k'\in K'$;

\item[(c)] $(\psi_{k'}(\phi_{j'}(0,i e_\Omega)))_{j'\in J',k'\in K'}$ is an $(R,4)$-lattice on $D$.
\end{itemize}
Observe that, by homogeneity, $(\psi_{k'}(\phi_{j'}(\zeta,z)))_{j'\in J',k'\in K'}$ is an $(R,4)$-lattice on $D$ for every $(\zeta,z)\in D$.
Then, let $(\zeta_{\ell},z_\ell)_{\ell\in L}$ be a family of elements of $B((0,i e_\Omega), 4 R)$ which is maximal for the property that the balls $B((\zeta_\ell,z_\ell),R\delta/4)$, as $\ell$ runs through $L$, are pairwise disjoint. Observe that clearly
\[
\card(L)\meg \frac{\nu_D(B((0,i e_\Omega),4 R+R\delta/4))}{\nu_D(B((0,i e_\Omega),R\delta/4))},
\]
so that $L$ is finite. Then, applying the preceding arguments to the measures $\mi_\ell\coloneqq \sum_{j',k'} \chi_{B( \psi_{k'}(\phi_{j'}(\zeta_\ell,z_\ell)) , R\delta )}\cdot \mi$, we find
\[
\sum_{\ell\in L} \sum_{j',k'} \Delta_\Omega^{p[\vect b+\vect d-2\vect{s}]}(h_k) \mi(B( \psi_{k'}(\phi_{j'}(\zeta_\ell,z_\ell)) , R\delta ))^p<\infty.
\]
Now, observe that $ B((0,i e_\Omega),4 R)\subseteq \bigcup_{\ell\in L}B((\zeta_\ell,z_\ell),R\delta/2)$, by maximality, and that
\[
M_{R\delta/2}(\mi)(\zeta,z)\meg M_{R\delta}(\mi)(\psi_{k'}(\phi_{j'}(\zeta_\ell,z_\ell)))
\]
for every $(\zeta,z)\in B((\zeta_\ell,z_\ell),R\delta/2)$ and for every $\ell \in L$.
In addition, there is a constant $C_5>0$ such that
\[
\frac{1}{C_5} \Delta_\Omega^{p[\vect b+\vect d-2\vect{s}]}(h)\meg \Delta_\Omega^{p[\vect b+\vect d-2\vect{s}]}(h')\meg C_5\Delta_\Omega^{p[\vect b+\vect d-2\vect{s}]}(h)
\]
for every $h,h'\in \Omega$ such that $d(h,h')<4 R$ (cf.~\cite[Corollary 2.49]{CalziPeloso}).
By the arbitrariness of $\ell$ and $(\zeta,z)$, this implies that
\[
\begin{split}
&\int_{B((0,i e_\Omega),4 R)} M_{R\delta/2}(\mi)^p ( \Delta_\Omega^{p[\vect b+\vect d-2\vect{s}]}\circ \rho)\,\dd \nu_D \meg C_5 M_{R\delta/2}(\nu_D)(0,i e_\Omega)  \sum_{\ell\in L}  M_{R\delta}(\mi)( (\zeta_\ell,z_\ell))^p,
\end{split}
\]
so that, by homogeneity,
\[
\begin{split}
&\int_{B( \psi_{k'}(\phi_{j'}(0,i e_\Omega)),4 R)} M_{R\delta/2}(\mi)^p ( \Delta_\Omega^{p[\vect b+\vect d-2\vect{s}]}\circ \rho)\,\dd \nu_D \\
&\qquad \qquad\meg C_5 M_{R\delta/2}(\nu_D)(0,i e_\Omega)  \sum_{\ell\in L} \Delta_\Omega^{p[\vect b+\vect d-2\vect{s}]}(h_k) M_{R\delta}(\mi)( \psi_{k'}(\phi_{j'}(\zeta_\ell,z_\ell)))^p.
\end{split}
\]
Therefore,
\[
\int_{D} M_{R\delta/2}(\mi)^p ( \Delta_\Omega^{p[\vect b+\vect d-2\vect{s}]}\circ \rho)\,\dd \nu_D <\infty,
\]
that is, $M_{R\delta/2}\in L^{p,p}_{(1-1/p)(\vect b+\vect d)-2\vect{s}}(D)$. The conclusion follows from~\cite[Lemma 5.1]{CPCarleson} in this case.

Finally, assume that $\mi$ is positive and that $p\Meg 1$. By means of Lemma~\ref{lem:8} and~\cite[Proposition 2.41]{CalziPeloso}, we may also assume that $B^{\vect{s'}}_{(\zeta,z)}\in A^{2,2}_{\vect{s''}}(D)$ for every $(\zeta,z)\in D$. Take $\vect{s_2}$ such that properties $\atomics^{2,2}_{\vect s,\vect{s_2}}$ and $\atomics^{2,2}_{\vect {s''},\vect{s_2}}$ hold (cf.~\cite[Crollary 5.11 and 5.16]{CalziPeloso}). Then, there is a  $(\delta,R)$-lattice $(\zeta_{j,k},z_{j,k})_{j\in J,k\in K}$ for some $\delta>0$ and some $R>1$, such that, if we define $A_1$ and $A_2$ as in the proof of the implication (2) $\implies$ (1), then both $A_1$ and $A_2$ are continuous (and onto). Define, in addition, $h_k\coloneqq \rho(\zeta_{j,k},z_{j,k})$ and $B_{j,k}\coloneqq B((\zeta_{j,k},z_{j,k}),R\delta)$ for every $(j,k)\in J\times K$.
Observe that, by~\cite[Theorem 2.47]{CalziPeloso}, there is a constant $C_6>0$ such that
\[
\frac{1}{C_6} \abs*{B^{\vect{s_2}}_{(\zeta,z)}(\zeta',z')}\meg \abs*{B^{\vect{s_2}}_{(\zeta,z)}(\zeta'',z'')}\meg C_6 \abs*{B^{\vect{s_2}}_{(\zeta,z)}(\zeta',z')}
\]
for every $(\zeta,z),(\zeta',z'),(\zeta'',z'')\in D$ such that $d((\zeta',z'),(\zeta'',z''))<R\delta$.

Then,
\[
\begin{split}
\sum_{j,k} \left( \Delta^{\vect b+\vect d- \vect{s}-\vect{s''}}_\Omega(h_k)\mi(B_{j,k})  \right)^p
&\meg C_6^{2 p} \sum_{j,k} \left( \Delta^{\vect b+\vect d- \vect s-\vect{s''}-2\vect{s_2}}_\Omega(h_k) \int_{B_{j,k}} \abs{B^{\vect{s_2}}_{(\zeta_{j,k},z_{j,k})}}^2\,\dd \mi  \right)^p\\
&\meg C_6^{2 p} \sum_{j,k} \left( \Delta^{\vect b+\vect d-\vect{s}-\vect{s''}-2\vect{s_2}}_\Omega(h_k) \langle B^{\vect{s_2}}_{(\zeta_{j,k},z_{j,k})}\bigg\vert B^{\vect{s_2}}_{(\zeta_{j,k},z_{j,k})}\rangle_{L^2(\mi)}  \right) ^p\\
&=c_{\vect{s'}}^p C_6^{2 p} \sum_{j,k}\abs*{ \langle T_{\mi,\vect{s'}}A_1 e_{j,k} \vert A_3 A_2 e_{j,k}\rangle_{L^{2,2}_{\vect b+\vect d-\vect {s'}-\vect{s''}}(D)}   }^p
\end{split}
\]
thanks to Proposition~\ref{prop:11}, where 
\[
A_3\colon A^{2,2}_{\vect{s''}}(D)\ni f \mapsto f (\Delta^{2 \vect {s''}+\vect{s'}-(\vect b+\vect d) }_\Omega\circ \rho)\in L^{2,2}_{\vect b+\vect d-\vect {s'}-\vect{s''}}(D)
\]
is a continuous linear mapping. Therefore, Propositions~\ref{prop:15} and~\ref{prop:22} imply that $\sum_{j,k} \left( \Delta^{\vect b+\vect d- \vect{s}-\vect{s''}}_\Omega(h_k)\mi(B_{j,k})  \right)^p$ is finite, so that the conclusion follows from~\cite[Lemma 5.1]{CPCarleson}.
\end{proof}

\section{Cesàro-type Operators}\label{sec:Cesaro}

In this section we study Cesàro-type operators, following (and extending) the definition given in~\cite{NanaSehba}. For these operators, continuity and compactness results basically follow from the corresponding results for Carleson measures (and this is rigorously so when the target space is of the form $A^{p,p}_\vs(D)$ with $p\in ]0,\infty[$). Using the relationship between Cesàro-type and Toeplitz operators, we then characterize the Cesàro-type operators which belong to some Schatten class $\Lin^p(A^{2,2}_{\vs_1}, \breve A^{2,2}_{\vs_2,\vs'}(D))$ (cf.~Definition~\ref{def:1} below).

\begin{deff}\label{def:1}
Take $\vect s\in \R^r$, $\vect{s'}\in \N_{\Omega'}$, and $p,q\in ]0,\infty]$. Define $\breve A^{p,q}_{\vect s,\vect{s'}}(D)$ as the Hausdorff space associated with the space of $f\in \Hol(D)$ such that $f*I^{-\vect{s'}}_\Omega\in A^{p,q}_{\vect s+\vect{s'}}(D)$, endowed with the induced topology. Define $\breve A^{p,q}_{\vect s,\vect{s'},0}(D)$ analogously. 
\end{deff}

Observe that the mapping $f\mapsto f*I^{-\vect{s'}}_\Omega$ induce isomorphisms of $\breve A^{p,q}_{\vect s,\vect{s'}}(D)$ and $\breve A^{p,q}_{\vect s,\vect{s'},0}(D)$ onto $A^{p,q}_{\vect s+\vect{s'}}(D)$ and $A^{p,q}_{\vect s+\vect{s'},0}(D)$, respectively (argue as in the proof of~\cite[Proposition 3.17]{CalziPeloso}). Therefore, $\breve A^{p,q}_{\vect s,\vect{s'},0}(D)\neq 0$ (resp.\ $\breve A^{p,q}_{\vect s,\vect{s'}}(D)\neq 0$) if and only if $\vect s+\vect{s'}\in \frac{1}{2 q}\vect m+(\R_+^*)^r$ (resp.\ $\vect s+\vect{s'}\in \R_+^r$ if $q=\infty$). Observe that $\breve A^{p,q}_{\vect s,\vect{s'}}(D)$ is canonically isomorphic to $\widehat A^{p,q}_{\vect s,\vect{s'}}(D)$ if $\vect{s'}$ is sufficiently large.

We observe explicitly that both $\breve A^{p,q}_{\vect
  s,\vect{s'}}(D)$ and $\widehat A^{p,q}_{\vect s,\vect{s'}}(D)$ can
be considered as possible generalization of the classical (weighted)
holomorphic Besov spaces (notice that almost no difference arises when
$r=1$, that is, when $D$ is a Siegel upper half-space or,
equivalently, biholomorphic to the unit
ball). Cf.~\cite{Zhu4,Zhu5,Zhu6} for more information of holomorphic
Besov spaces on \emph{bounded} symmetric domains.

For every $\vect{s'}\in\N_{\Omega'}$ and for every $g\in \Hol(D)$, define a mapping $\Cc_{g,\vect{s'}}\colon \Hol(D)\to \Hol(D)/\ker(\,\cdot\,*I^{-\vect{s'}}_\Omega)$ by
\[
\Cc_{g,\vect{s'}}(f)*I^{-\vect{s'}}_\Omega\coloneqq f (g*I^{-\vect{s'}}_\Omega)
\]
for every $f\in \Hol(D)$.
Observe that saying that $\Cc_{g,\vs'}$ maps $A^{p_1,q_1}_{\vs_1}(D)$ into $\breve A^{p,p}_{\vs_2}(D)$ continuously (resp.\ compactly) is equivalent to saying that $(g*I^{-\vs'}_\Omega)(\Delta^{p\vs_2-(\vb+\vd)}_\Omega\circ \rho)\cdot \nu_D$ is a $p$-Carleson measure (resp.\ a vanishing (or compact) $p$-Carleson measure) for $A^{p_1,q_1}_{\vs_1}(D)$. In the particular case $p_2=q_2=p$, then, the following Propositions~\ref{prop:5},~\ref{prop:5bis}, and~\ref{prop:9} follow from~\cite[Proposition 5.2 and 5.3, and Theorem  5.5]{CPCarleson}. Even though the proofs in the general case are analogous, we repeat them for the sake of completeness.

\begin{prop}\label{prop:4}\label{prop:5}
Take $p_1,p_2,q_1,q_2\in ]0,\infty]$, $\vect {s_1},\vect{s_2}\in \R^r$ and $\vect{s'}\in \N_{\Omega'}$. Define $p_3\coloneqq \big(\frac{1}{p_2}-\frac{1}{p_1}\big)_+^{-1}$, $q_3\coloneqq \big(\frac{1}{q_2}-\frac{1}{q_1}\big)_+^{-1}$, and $\vect{s_3}\coloneqq \vect{s_2}-\vect{s_1}+\big(\frac{1}{p_1}-\frac{1}{p_2}\big)_+(\vect b+\vect d) $. 
Take $g\in \breve A^{p_3,q_3}_{\vect{s_3},\vect{s'}}(D)$ (resp.\ $g\in \breve A^{p_3,q_3}_{\vect{s_3},\vect{s'},0}(D)$).
Then, $\Cc_{g,\vect{s'}}$ induces  continuous (resp.\ compact) linear mappings $A^{p_1,q_1}_{\vect {s_1},0}(D)\to \breve A^{p_2,q_2}_{\vect{s_2},\vect{s'},0}(D)$  and $A^{p_1,q_1}_{\vect {s_1}}(D)\to \breve A^{p_2,q_2}_{\vect{s_2},\vect{s'}}(D)$.
\end{prop}

This extends one implication of~\cite[Corollary 3.9 and Theorem 3.11]{NanaSehba}, where the case in which $p_1=q_1\in [1,\infty[$, $p_2=q_2\in [1,\infty]$, $\vect{s_1}=\vect{s_2}\in \R\vect 1_r$, and $D$ in an irreducible symmetric tube domain, is considered. 

\begin{proof}
\textsc{Step I.} Assume that $g\in \breve A^{p_3,q_3}_{\vect{s_3}}(D)$. By~\cite[Theorem 3.23]{CalziPeloso}, given a $(\delta,R)$-lattice  $(\zeta_{j,k},z_{j,k})_{j\in J,k\in K}$ on $D$ for some $\delta>0$ and $R>1$ (cf.~\cite[Lemma 2.55]{CalziPeloso}), we may find a constant $C_1>0$ such that the operators $S_1,S_2,S_3\colon \Hol(D)\to \C^{J\times K}$ defined by
\[
(S_\ell f)_{j,k}\coloneqq \Delta_\Omega^{\vect{s_\ell}-(\vect b+\vect d)/p_\ell}(\rho(\zeta_{j,k},z_{j,k})) \max_{\overline B((\zeta_{j,k},z_{j,k}),R\delta)} \abs{f},
\]
for every $f\in \Hol(D)$ and for $\ell=1,2,3$, satisfy
\[
\frac{1}{C_1}\norm{S_\ell f}_{\ell^{p_\ell,q_\ell}(J,K)}\meg  \norm{f}_{A^{p_\ell,q_\ell}_{\vect{s_\ell}}(D)}\meg C_1\norm{S_\ell f}_{\ell^{p_\ell,q_\ell}(J,K)}
\]
for every $f\in \Hol(D)$ and for $\ell=1,2,3$. We may also assume that $f\in A^{p_\ell,q_\ell}_{\vect{s_\ell},0}(D)$ if and only if $S_\ell f\in \ell^{p_\ell,q_\ell}_0(J,K)$. 
Then,
\[
\begin{split}
\norm{\Cc_{g,\vect{s'}}(f)}_{A^{p_2,q_2}_{\vect{s_2},\vect{s'}}(D)}&\meg C_1 \norm{S_2(f (g*I^{-\vect{s'}}_\Omega))}_{\ell^{p_2,q_2}(J,K)}\\
&\meg C_1 \norm{S_1(f) S_3(g*I^{-\vect{s'}}_\Omega)}_{\ell^{p_2,q_2}(J,K)}\\
&\meg C_1 \norm{S_1(f)}_{\ell^{p_1,q_1}(J,K)} \norm{S_3(g*I^{-\vect{s'}}_\Omega)}_{\ell^{p_3,q_3}(J,K)}\\
&\meg C_1^3 \norm{f}_{A^{p_1,q_1}_{\vect{s_1}}(D)}\norm{g}_{\breve A^{p_3,q_3}_{\vect{s_3}}(D)}
\end{split}
\]
for every $f\in \Hol(D)$. The assertion follows in this case.

\textsc{Step II.} Assume that $g\in \breve A^{p_3,q_3}_{\vect{s_3},\vect{s'},0}(D)$. By~\cite[Theorem 3.22]{CalziPeloso}, we may find a $(\delta,R)$-lattice  $(\zeta_{j,k},z_{j,k})_{j\in J,k\in K}$ on $D$ for some $\delta>0$ and $R>1$, and a constant $C_2>0$ such that the operators $S_1,S_2,S_3\colon \Hol(D)\to \C^{J\times K}$ defined by
\[
(S_\ell f)_{j,k}\coloneqq \Delta_\Omega^{\vect{s_\ell}-(\vect b+\vect d)/p_\ell}(\rho(\zeta_{j,k},z_{j,k})) f(\zeta_{j,k},z_{j,k}),
\]
for every $f\in \Hol(D)$ and for $\ell=1,2,3$, satisfy
\[
\frac{1}{C_2}\norm{S_\ell f}_{\ell^{p_\ell,q_\ell}(J,K)}\meg  \norm{f}_{A^{p_\ell,q_\ell}_{\vect{s_\ell}}(D)}\meg C_2\norm{S_\ell f}_{\ell^{p_\ell,q_\ell}(J,K)}
\]
for every $f\in A^{p_\ell,q_\ell}_{\vect{s_\ell}}(D)$ and for $\ell=1,2,3$. Define $X_{\ell,0}\coloneqq S_\ell(A^{p_\ell,q_\ell}_{\vect{s_\ell},0}(D))$ and $X_\ell\coloneqq S_\ell(A^{p_\ell,q_\ell}_{\vect{s_\ell}}(D))$ for $\ell=1,2$, and define $\lambda_{g,\vect{s'}}\coloneqq S_3(g*I^{-\vect{s'}}_\Omega)$.
By \textsc{step I}, it will suffice to prove that the continuous linear mappings
\[
X_{1,0}\ni \lambda \mapsto \lambda \lambda_{g,\vect{s'}}\in X_{2,0} \qquad \text{and} \qquad X_1\ni \lambda \mapsto \lambda \lambda_{g,\vect{s'}}\in X_2
\]
are compact, where the product is defined componentwise. Since $\lambda_{g,\vect{s'}}\in \ell^{p_3,q_3}_0(J,K)$, and since $X_\ell$ has the topology induced by $\ell^{p_\ell,q_\ell}(J,K)$ for $\ell=1,2$, the assertion follows easily.\footnote{Indeed, if $(Y_\beta)_{\beta\in \N}$ is an increasing sequence of finite subsets of $J\times K$ whose union covers $J\times K$, then the operator of multiplication by $\lambda_{g,\vect{s'}}$ is the limit of the operators of multiplication by $\chi_{Y_\beta}\lambda_{g,\vect{s'}}$, which have finite rank and are therefore compact, in $\Lin(X_1;X_2)$. }
\end{proof}

\begin{prop}\label{prop:5bis}
Take $p_1,q_1,q_2\in ]0,\infty]$, $\vect {s_1},\vect{s_2}\in \R^r$ and $\vect{s'}\in \N_{\Omega'}$. Assume   that $\vect{s_1}\in \frac{1}{2 q_1}\vect m+(\R_+^*)^r$ (resp.\ $\vect{s_1}\in \R_+^r$ if $q_1=\infty$). Define $\vs_3\coloneqq \vs_2-\vs_1+\big(\frac{1}{p_1}-\frac{1}{p_2})(\vb+\vd)$ and
take $g\in \Hol(D)$ so that $\Cc_{g,\vect{s'}}$ induces a  continuous linear mapping $A^{p_1,q_1}_{\vect {s_1},0}(D)\to \breve A^{q_2,q_2}_{\vect{s_2},\vect{s'},0}(D)$ (resp.\ $A^{p_1,q_1}_{\vect {s_1}}(D)\to \breve A^{q_2,q_2}_{\vect{s_2},\vect{s'}}(D)$).Then, $g\in \breve A^{\infty,\infty}_{\vs_3,\vect{s'}}(D)$. 

If, in addition, $\vs\in (\R_+^*)^r$ when $p_1=q_1=\infty$ and $\Cc_{g,\vect{s'}}$ induces a  compact linear mapping $A^{p_1,q_1}_{\vect {s_1},0}(D)\to \breve A^{q_2,q_2}_{\vect{s_2},\vect{s'},0}(D)$ (resp.\ $A^{p_1,q_1}_{\vect {s_1}}(D)\to \breve A^{q_2,q_2}_{\vect{s_2},\vect{s'}}(D)$), then $g\in \breve A^{\infty,\infty}_{\vs_3,\vect{s'}}(D)$. 
\end{prop}

\begin{proof}
Take $\vect{s_4}\in \frac{1}{p_1}(\vect b+\vect d)-\frac{1}{2 p_1} \vect{m'}-(\R_+^*)^r$ such that $\vect{s_1}+\vect{s_4}\in\frac{1}{p_1}(\vect b+\vect d)-\frac{1}{2 q_1} \vect{m'}-(\R_+^*)^r$, so that~\cite[Proposition 2.41]{CalziPeloso} shows that there is a constant $C_1>0$ such that
\[
\norm{B^{\vect{s_4}}_{(\zeta,z)}}_{A^{p_1,q_1}_{\vect{s_1}}(D)}=C_1 \Delta^{\vect{s_1}+\vect{s_4}-(\vect b+\vect d)/p_1}_\Omega(\rho(\zeta,z))
\]
for every $(\zeta,z)\in D$. In addition, by~\cite[Theorem 2.47]{CalziPeloso}, there is a constant $C_2>0$ such that 
\[
\abs{B^{\vect {s_4}}_{(\zeta,z)}(\zeta',z')}\Meg C_2 \Delta^{\vect {s_4}}_\Omega(\rho(\zeta,z))
\]
for every $(\zeta,z),(\zeta',z')\in D$ such that $d((\zeta,z),(\zeta',z'))\meg 1$.
Therefore, there is a constant $C_3>0$ such that
\[
\begin{split}
C_3\Delta^{\vect{s_1}+\vect{s_4}-(\vect b+\vect d)/p_1}_\Omega(\rho(\zeta,z))&\Meg \norm{\Cc_{g,\vect{s'}}(B^{\vect{s_4}}_{(\zeta,z)})}_{\breve A^{p_2,q_2}_{\vect{s_2}(D)}}\\
&=\norm{ B^{\vect{s_4}}_{(\zeta,z)} (g*I^{-\vect{s'}}_\Omega) }_{A^{p_2,q_2}_{\vect{s_2}(D)}}\\
&\Meg C_2\Delta^{\vect {s_4}}_\Omega(\rho(\zeta,z)) \norm{\chi_{B((\zeta,z),1)}(g*I^{-\vect{s'}}_\Omega)}_{L^{p_2,q_2}_{\vect{s_2}(D)}}
\end{split}
\]
for every $(\zeta,z)\in D$. Now, by a simple modification of~\cite[Lemma 3.24]{CalziPeloso}, we may find two constants $\rho,C_4>0$, with $\rho\meg 1$, such that
\[
\Delta_\Omega^{\vect{s_2}-(\vect b+\vect d)/p_2}(\rho(\zeta,z))\abs{h(\zeta,z)}\meg C_4 \norm{\chi_{B((\zeta,z),\rho)} h}_{L^{p_2,q_2}_{\vect{s_2}(D)}}
\]
for every $h\in \Hol(D)$. Therefore,
\[
\Delta_\Omega^{\vect{s_3}}(\rho(\zeta,z))\abs{(g*I^{-\vect{s'}}_\Omega)(\zeta,z)}\meg \frac{C_3 C_4}{C_2}
\]
for every $(\zeta,z)\in D$, so that $g\in \breve A^{\infty,\infty}_{\vect{s_3},\vect{s'}}(D)$.

Now, assume that  $\Cc_{g,\vect{s'}}$ induces a  compact linear mapping $A^{p_1,q_1}_{\vect {s_1},0}(D)\to \breve A^{p_2,q_2}_{\vect{s_2},\vect{s'}}(D)$. Arguing as in the proof of~\cite[Proposition 5.3]{CPCarleson}, we may find $\vect{s_5}\in \R^r$ such that
\[
\lim_{(\zeta,z)\to \infty}\Delta_\Omega^{(\vect b+\vect d)/p_1-\vect{s_1}-\vect{s_5}}(\rho(\zeta,z))\norm{ \Cc_{g,\vect{s'}} B^{\vect{s_5}}_{(\zeta,z)}}_{\breve A^{p_2,q_2}_{\vect{s_2}}(D)}=0.
\]
Then, by means of the preceding estimates, it is readily seen that  $g\in \breve A^{\infty,\infty}_{\vect{s_3},\vect{s'},0}(D)$.
\end{proof}

\begin{prop}\label{prop:9}
Take $p_1,q_1,q_2\in ]0,\infty]$, $\vect {s_1},\vect{s_2}\in \R^r$ and $\vect{s'}\in \N_{\Omega'}$. Assume   that $\vect{s_1}\in \frac{1}{2 q_1}\vect m+(\R_+^*)^r$ (resp.\ $\vect{s_1}\in \R_+^r$ if $q_1=\infty$) and that property $\atomic^{p_1,q_1}_{\vect s,\vect{s''},0}$ holds for some $\vect{s''}\in \R^r$.
Define $p_3,q_3$, and  $\vect{s_3}$ as in Proposition~\ref{prop:4} (with $p_2\coloneqq q_2$), and take $g\in \Hol(D)$ so that $\Cc_{g,\vect{s'}}$ induces a  continuous linear mapping $A^{p_1,q_1}_{\vect {s_1},0}(D)\to \breve A^{q_2,q_2}_{\vect{s_2},\vect{s'},0}(D)$ (resp.\ $A^{p_1,q_1}_{\vect {s_1}}(D)\to \breve A^{q_2,q_2}_{\vect{s_2},\vect{s'}}(D)$).Then, $g\in \breve A^{p_3,q_3}_{\vect{s_3},\vect{s'}}(D)$. 
\end{prop}

This extends one implication of~\cite[Corollary 3.9 (ii)]{NanaSehba}, where the case in which $p_1=q_1,q_2\in [1,\infty[$,  $\vect{s_1}=\vect{s_2}\in \R\vect 1_r$, and $D$ in an irreducible symmetric tube domain, is considered. 

\begin{proof}
If $q_2=\infty$, then the assertion follows from Proposition~\ref{prop:5bis}, so that we may assume that $q_2<\infty$.
Let $\mi$ be the measure on $D$ with density
\[
(\zeta,z)\mapsto \Delta_\Omega^{q_2\vect s+\vect d}(\rho(\zeta,z)) \abs{(g*I^{-\vect{s'}}_\Omega)(\zeta,z)}^{q_2}
\]
with respect to Lebesgue measure. Then, the assumption means that $A^{p_1,q_1}_{\vect {s_1},0}(D)$ embeds continuously into $L^{q_2}(\mi)$. Thus,~\cite[Theorem 5.5]{CPCarleson} implies that the function 
\[
(\zeta,z)\mapsto \mi(B(((\zeta,z),1))= \norm{\chi_{B((\zeta,z),1)} (g*I^{-\vect{s'}}_\Omega) }_{L^{q_2,q_2}_{\vect{s_2}}(D)}^{q_2}
\] 
belongs to $L^{(p_1/q_2)',(q_1/q_2)'}_{\max(p_1,q_2)(\vect b+\vect d)/p_1-q_2 \vect{s_1}}(D)$. Then, arguing as in the proof of Proposition~\ref{prop:5bis}  we may find two constants $\rho,C>0$, with $\rho\meg 1$, such that
\[
\Delta_\Omega^{\vect{s_2}-(\vect b+\vect d)/q_2}\abs{h(\zeta,z)}\meg C \norm{\chi_{B((\zeta,z),\rho)} h}_{L^{q_2,q_2}_{\vect{s_2}(D)}}
\]
for every $h\in \Hol(D)$. Therefore,  $g\in A^{p_3,q_3}_{\vect{s_3},\vect{s'}}(D)$.
\end{proof}

\begin{prop}\label{prop:29}
Take $p\in ]0,\infty[$, $\vect{s_1},\vect{s_2}\in\frac{1}{2 p}\vect m+(\R_+^*)^r$, $\vect{s'}\in \N_{\Omega'}$, and $g\in \breve A^{\infty,\infty}_{\vect{s_2}-\vect{s_1},\vect{s'}}(D)$. Then, the following conditions are equivalent:
\begin{enumerate}
\item[{\em(1)}] $\Cc_{g,\vect{s'}}$ induces an isomorphism of $A^{p,p}_{\vect{s_1}}(D)$ onto a closed subspace of $A^{p,p}_{\vect{s_2}}(D)$;

\item[{\em(2)}] there are $\eps,R,C>0$ such that 
\[
\nu_D\left(\Set{(\zeta',z')\in B((\zeta,z),R)\colon \abs{(g*I^{-\vect{s'}}_\Omega)(\zeta',z')}> \eps \Delta_\Omega^{\vect{s_1}-\vect{s_2}-\vect{s'}}(\rho(\zeta',z'))   } \right)\Meg C
\]
for every $(\zeta,z)\in D$.
\end{enumerate}  
\end{prop}

This extends~\cite[Corollary 1]{Luecking}, which deals with the case
in which $\vect{s'}=0$, $\vect{s_1}=\vect{s_2}=-\vect d/p$, and $D$ is
the unit disc in $\C$ (but $g$ is a bounded \emph{measurable}
function).  
\begin{proof}
(1) $\implies$ (2). Define 
\[
G_\eps\coloneqq \Set{(\zeta,z)\in D\colon \abs{(g*I^{-\vect{s'}}_\Omega)(\zeta,z)}>\eps \Delta_\Omega^{\vect{s_1}-\vect{s_2}-\vect{s'}}(\rho(\zeta,z))}
\]
for every $\eps>0$. Observe that, by assumption, there is a constant $C_1>0$ such that
\[
\norm{\Cc_{g,\vect{s'}}f}_{\breve A^{p,p}_{\vect {s_2},\vect{s'}}(D)}\Meg C_1 \norm{f}_{A^{p,p}_{\vect{s_1}}(D)}
\] 
for every $f\in A^{p,p}_{\vect{s_1}}(D)$, so that
\[
C_1\norm{f}_{A^{p,p}_{\vect{s_1}}(D)}\meg \norm{f (g*I^{-\vect{s'}}_\Omega)  }_{A^{p,p}_{\vect{s_2}+\vect{s'}}(D)}
\]
for every $f\in A^{p,p}_{\vect{s_1}(D)}$. If we define $C_2\coloneqq \norm{g}_{\breve A^{\infty,\infty}_{\vect{s_2}-\vect{s_1},\vect{s'}}(D)  }$, then
\[
C_1\norm{f}_{A^{p,p}_{\vect{s_1}}(D)}\meg C_2\norm{\chi_{G_\eps}f}_{A^{p,p}_{\vect{s_1}}(D)}+\eps \norm{\chi_{D\setminus G_\eps} f}_{A^{p,p}_{\vect{s_1}}(D)}
\]
for every $\eps>0$ and for every $f\in A^{p,p}_{\vect{s_1}(D)}$. 

Now, assume by contradiction that (2) is not satisfied. Then,~\cite[Theorem 7.3]{CPCarleson} implies that for every $j\in\N^*$ there is $f_j\in A^{p,p}_{\vect s+\vect{s'}}(D)$ such that $\norm{f_j}_{A^{p,p}_{\vect{s_1}}(D)}=1$ and such that
\[
\norm{\chi_{G_{1/j}}f_j}_{A^{p,p}_{\vect{s_1}}(D)}\meg 1/j.
\]
Then,
\[
C_2 \meg \lim_{j\to \infty} \left(  C_2\norm{\chi_{G_{1/j}}f_{1/j}}_{A^{p,p}_{\vect{s_1}}(D)}+\eps \norm{\chi_{D\setminus G_{1/j}} f_j}_{A^{p,p}_{\vect{s_1}}(D)}\right) =0,
\]
which is absurd.

(2) $\implies$ (1). By~\cite[Theorem 7.3]{CPCarleson}, there is a constant $C_3>0$ such that
\[
\norm{f}_{A^{p,p}_{\vect{s_1}}(D)}\meg C_3\norm{\chi_{G_\eps}f}_{A^{p,p}_{\vect{s_1}}(D)}
\]
for every $f\in A^{p,p}_{\vect{s_1}}(D)$, so that
\[
\begin{split}
\norm{\Cc_{g,\vect{s'}}f}_{\breve A^{p,p}_{\vect {s_2},\vect{s'}}(D)}&= \norm{f(g*I^{-\vect{s'}}_\Omega)}_{A^{p,p}_{\vect{s_2},\vect{s'}}(D)}\\
&\Meg \eps \norm{\chi_{G_\eps} f}_{A^{p,p}_{\vect{s_1}}(D)}\\
&\Meg \frac{\eps}{C_3} \norm{f}_{A^{p,p}_{\vect{s_1}}(D)}
\end{split}
\]
for every $f\in A^{p,p}_{\vect{s_1}}(D)$. The conclusion follows by means of Proposition~\ref{prop:4}.
\end{proof}

\begin{teo}\label{prop:30}
Take $\vect{s_1}\in \frac 1 4 \vect m+(\R_+^*)^r$, $\vect{s_2}\in \R^r$, $\vect{s'}\in \N_{\Omega'}$, $p\in ]0,\infty[$, and $g\in \Hol(D)$. Then, $\Cc_{g,\vect{s}}\in \Lin^p(A^{2,2}_{\vect {s_1}}(D);\breve A^{2,2}_{\vect{s_2},\vect{s'}}(D))$ if and only if $g\in \breve A^{p,p}_{\vect{s_2}-\vect{s_1}+(\vect b+\vect d)/p,\vect{s'}}(D)$.
\end{teo}

This extends one implication of~\cite[Theorem 4.3]{NanaSehba}, where the case in which $\vect{s_1}=\vect{s_2}\in \R\vect 1_r$, $p\Meg 2$, and $D$ in an irreducible symmetric tube domain, is considered. 

\begin{proof}
Let $\mi$ be the measure on $D$ with density $ (\Delta_\Omega^{2 \vect {s_2}+2 \vect{s'}-(\vect b+\vect d)}\circ \rho)\abs{g*I^{-\vect{s'}}_\Omega}^2$ with respect to $\nu_D$. Define $\vect{s''}\coloneqq \vect b+\vect d-2 \vect{s_1}$. 
Notice that $\Cc_{g,\vect{s'}}\in \Lin(A^{2,2}_{\vect {s_1}}(D);\breve A^{2,2}_{\vect {s_2}}(D))$ if and only if $g\in \breve A^{\infty,\infty}_{\vect{s_2}-\vect{s_1},\vect{s'}}(D)$, thanks to Propositions~\ref{prop:4} and~\ref{prop:5bis}, and that  $T_{\mi,\vect{s''}}\in \Lin(A^{2,2}_{\vect{s_1}}(D)) $ if and only if $M_1(\mi)\in L^{\infty,\infty}_{\vect{s''}}(D)$, thanks to Proposition~\ref{prop:11}. Therefore, arguing as in the proof of Proposition~\ref{prop:9}, we see that $\Cc_{g,\vect{s'}}\in \Lin(A^{2,2}_{\vect {s_1}}(D);\breve A^{2,2}_{\vect {s_2}}(D))$ if and only $T_{\mi,\vect{s''}}\in \Lin(A^{2,2}_{\vect{s_1}}(D)) $. Hence, we may assume that both of these conditions hold.
Since we may also assume that $g\neq 0$, we deduce that $\vect{s_2}+\vect{s'}\in\frac{1}{4}\vect m+(\R_+^*)^r$.

Then,  for every $f\in A^{2,2}_{\vect{s_1}}(D)$,
\[
\begin{split}
\norm{\Cc_{g,\vect{s'}} f  }_{\breve A^{2,2}_{\vect{s_2},\vect{s'}}(D)}^2&= \norm{f  (g*I^{-\vect{s'}}_\Omega) }_{A^{2,2}_{\vect{s_2}+\vect{s'}}(D)}^2\\
&=\norm{f}_{L^2(\mi)}^2\\
&=c_{\vect{s''}}\abs{\langle T_{\mi,\vect{s''}} f\vert f\rangle_{A^{2,2}_{\vect{s_1}}(D)}}  ,
\end{split}
\]
where the last equality follows from Proposition~\ref{prop:11}. 
Then, $T_{\mi,\vect{s''}} $ is (self-adjoint and) positive, so that
\[
\norm{\Cc_{g,\vect{s'}} f  }_{\breve A^{2,2}_{\vect{s_2},\vect{s'}}(D)}=c_{\vect{s''}}^{1/2} \norm{T_{\mi,\vect{s''}}^{1/2}}_{A^{2,2}_{\vect{s_1}}(D)}
\]
for every $f\in A^{2,2}_{\vect{s_1}}(D)$.
Hence, there is a linear isometry of $\Cc_{g,\vect{s'}}(A^{2,2}_{\vect{s_1}}(D))$ onto $T_{\mi,\vect{s''}}^{1/2}(A^{2,2}_{\vect{s_1}}(D))$ such that
\[
U\Cc_{g,\vect{s'}} =c_{\vect{s''}}^{1/2}T_{\mi,\vect{s''}}^{1/2},
\]
so that Proposition~\ref{prop:15} implies that
\[
\norm{\Cc_{g,\vect{s'}} }_{\Lin^p(A^{2,2}_{\vect {s_1}}(D);\breve A^{2,2}_{\vect{s_2},\vect{s'}}(D))}=c_{\vect{s''}}^{1/2} \norm{T_{\mi,\vect{s''}}}^2_{\Lin^{p/2}(A^{2,2}_{\vect {s_1}}(D))}.
\]
The assertion then follows from Theorem~\ref{prop:33}.
\end{proof}

\section{Schatten Classes}\label{sec:Schatten}

In this section we review some basic fact about Schatten classes of
linear mappings between two hilbertian spaces  (cf.~Definition~\ref{def:Schatten-classes}).

\begin{prop}\label{prop:25bis}
Let $H_1,H_2$ be two hilbertian spaces,  and $T\in \Lin(H_1;H_2)$. Then, the mapping
\[
]0,\infty]\ni p\mapsto \norm{T}_{\Lin^p(H_1;H_2)}\in [0,\infty]
\]
is decreasing.
\end{prop}

This result follows easily from the singular value decomposition of compact operators (cf., e.g.,~\cite[p.~261]{Kato}).

\begin{prop}\label{prop:15}
Let $H_1,H_1',H_2,H_2'$ be four hilbertian spaces, and take $U\in \Lin(H_1';H_1), T\in \Lin(H_1;H_2)$, and $V\in \Lin(H_2;H_2')$ and $p\in ]0,\infty]$. Then, the following hold:
\begin{enumerate}
\item[{\textrm{(1)}}] if $T\in\Lin^p_0(H_1,H_2)$ (resp.\ $T\in\Lin^p(H_1,H_2)$), then $V T U\in \Lin^p_0(H_1';H_2')$ (resp.\ $V T U\in \Lin^p(H_1';H_2')$)
and
\[
\norm{VTU}_{\Lin^p(H_1';H_2')}\meg \norm{V}_{\Lin(H_2;H_2')} \norm{T}_{\Lin^p(H_1;H_2)} \norm{U}_{\Lin(H_1';H_1)};
\]

\item[{\textrm{(2)}}]  if $U$ is onto, $V$ is an isomorphism onto its image, and $ V T U\in \Lin^p_0(H_1';H_2')$ (resp.\  $V T U\in \Lin^p(H_1';H_2')$), then $T\in\Lin^p_0(H_1,H_2)$ (resp.\ $T\in\Lin^p(H_1,H_2)$);

\item[{\textrm{(3)}}] if $U^*$ is an isometry on $T^*(H_2)$, and $V$ is an isometry on $T(H_1)$, then
\[
\norm{VTU}_{\Lin^p(H_1';H_2')}=\norm{T}_{\Lin^p(H_1;H_2)}.
\]
\end{enumerate}
\end{prop}

This follows easily from~\cite[Proposition 1.30 and Corollary 1.35]{Zhu}.

\begin{prop}\label{prop:22}
Let $H$ be a hilbertian space, $p\in [1,\infty[$, and $T\in \Lin(H)$. Then, $T\in \Lin^p(H)$ if and only if
\[
\sum_{j\in J} \abs{\langle T v_j\vert v_j\rangle}^p<\infty
\]  
for every countable  orthonormal family $(v_j)_{j\in J}$  of elements of $H$. In addition,
\[
\frac{1}{2^{1/p' +(1/2-1/p)_+}}\norm{T}_{\Lin^p(H)}^p\meg \sup_{(v_j)} \sum_{j}\abs{\langle T v_j\vert v_j\rangle}^p\meg \norm{T}_{\Lin^p(H)}^p,
\]
where $(v_j)_{j\in J}$ runs through the set of orthonormal finite families of elements of $H$. If $T$ is normal, then equality holds.
\end{prop}

This follows from~\cite[Theorem 1.27]{Zhu} and~\cite[Lemma 2.3.4]{Ringrose}. Use the spectral theorem (applied to the self-adjoint operators $T+T*$ and $(T-T^*)/i$) to prove that $T$ is compact provided that $\sum_{j\in J} \abs{\langle T v_j\vert v_j\rangle}^p<\infty$ for every countable  orthonormal family $(v_j)_{j\in J}$  of elements of $H$.

\begin{prop}\label{prop:13}
Let $H_1,H_2$ be two hilbertian spaces,  $T\in \Lin(H_1;H_2)$, and $p\in [1,\infty[$. Then, $T\in \Lin^p(H_1;H_2)$ if and only if 
\[
\sum_{j\in J} \abs{\langle T v_j\vert w_j\rangle}^p<\infty
\]  
for every countable set $J$ and for every two orthonormal families $(v_j)_{j\in J}$ and $(w_j)_{j\in J}$ of elements of $H_1$ and $H_2$, respectively. In addition,
\[
\norm{T}_{\Lin^p(H_1;H_2)}^p=\sup_{(v_j),(w_j)} \sum_{j}\abs{\langle T v_j\vert w_j\rangle}^p,
\]
where $(v_j)_{j\in J}$ and $(w_j)_{j\in J}$  run through the sets of orthonormal finite families of elements of $H_1$ and $H_2$, respectively.
\end{prop}

This follows from~\cite[Lemma 2.3.4]{Ringrose} and~\cite[Theorem 1.28]{Zhu}, as for Proposition~\ref{prop:22}.

\begin{prop}\label{prop:34}
Let $H_1$ and $H_2$ be two hilbertian spaces, $p\in ]0,2]$, $T\in \Lin(H_1;H_2)$, and $(v_j)_{j\in J}$, $(w_k)_{k\in K}$ two orthonormal bases of $H_1$ and $H_2$, respectively. Then,
\[
\norm{T}_{\Lin^p(H_1;H_2)}^p\meg \sum_{j\in J} \norm{T v_j}^p\meg \sum_{j\in J,k\in K} \abs{\langle T v_j\vert w_k\rangle}^p.
\]
\end{prop}

 This follows from~\cite[p.~95]{GohbergKrein} using the polar decomposition of $T$.

\begin{prop}\label{prop:17}
Let $H_1,H_2$ be two hilbertian spaces, and take $p\in [1,\infty]$. Then, the mapping $(T,S)\mapsto \norm{T-S}_{\Lin^p(H_1;H_2)}$ is a pseudo-distance which is lower semi-continuous in the weak topology of $\Lin_s(H_1;H_2)$\footnote{We denote by $\Lin_s(H_1;H_2)$ the space $\Lin(H_1;H_2)$ endowed with the topology of simple (or pointwise) convergence. A pseudo-distance on a set $X$ is a symmetric mapping $d\colon X\times X\to [0,\infty]$ which vanishes on the diagonal and satisfies the triangular inequality.} which endows $\Lin^p(H_1;H_2)$ with the topology of a Banach space. In addition, $\Lin^p_0(H_1;H_2)$ is the closure in $\Lin^p(H_1;H_2)$ of the space of linear operators with finite rank.
\end{prop}

This follows from Propositions~\ref{prop:13} and~\ref{prop:25bis}, using the singular value decomposition of compact operators (cf., e.g.,~\cite[p.~261]{Kato}).

\begin{prop}\label{prop:16}
Let $H_1,H_2,H_3$ be three hilbertian spaces, and take $p,q,r\in [1,\infty[$ such that $\frac{1}{r}=\frac{1}{p}+\frac{1}{q}$. Then, for every $T\in \Lin^p(H_1;H_2)$  and for every $S\in \Lin^q(H_2;H_3)$, one has $S T\in \Lin^r(H_1;H_3)$, and
\[
\norm{ST}_{\Lin^{r}(H_1;H_3)}\meg \norm{S}_{\Lin^q(H_2;H_3)}\norm{T}_{\Lin^p(H_1;H_2)}.
\]
\end{prop}

This follows from~\cite[Theorem 2.3.10]{Ringrose} using the polar decompositions of $T$ and $S$.

\begin{prop}\label{prop:18}
Let $H_1,H_2$ be two hilbertian spaces, and take $p\in [1,\infty]$. Then, the sesquilinear mapping
\[
\Lin^p_0(H_1;H_2)\times \Lin^{p'}(H_1;H_2)\ni (T,S)\mapsto \tr(S^* T)\in \C
\]
induces an isometric antilinear isomorphism of $\Lin^{p'}(H_1;H_2)$ onto $\Lin^p_0(H_1;H_2)'$.
\end{prop}

This extends~\cite[Theorem 2.3.12]{Ringrose} and is proved similarly.

\begin{prop}\label{prop:19}
Let $H_1,H_2$ be two hilbertian spaces, and take $p\in [1,\infty[$. Then,
\[
(\Lin^1(H_1;H_2), \Lin^\infty_0(H_1;H_2))_{[1/p']}=(\Lin^1(H_1;H_2), \Lin^\infty(H_1;H_2))_{[1/p']}\cong \Lin^p(H_1;H_2).
\]
\end{prop}

Here, $(A,B)_{[\theta]}$ denotes  the complex interpolation space of the Banach pair $(A,B)$ of index $\theta\in [0,1]$.

This follows from~\cite[Theorem 13.1 of Chapter III]{GohbergKrein} using standard techniques.

\end{document}